\documentclass[reqno]{amsart}
\usepackage[left=1in,right=1in,top=1in,bottom=1in]{geometry}
 \usepackage{microtype}
\usepackage{hyperref}
         
\usepackage{amsmath}             
\usepackage{amsfonts}             
\usepackage{amsthm}               
\usepackage{amsbsy}
\usepackage{amssymb}
\usepackage{amsthm}
\usepackage{amsrefs}

\usepackage{enumerate}

\newtheorem{thm}{Theorem}[section]
\newtheorem{lem}[thm]{Lemma}

\theoremstyle{definition}
\newtheorem{defn}[thm]{Definition}

\newtheorem{rem}[thm]{Remark}
\newtheorem{exam}[thm]{Example}
\newtheorem{cons}[thm]{Construction}

\newcommand{\bC}{{\mathbb{C}}}
\newcommand{\bN}{{\mathbb{N}}}
\newcommand{\bR}{{\mathbb{R}}}

\newcommand{\A}{{\mathcal{A}}}

\newcommand{\F}{{\mathcal{F}}}
\renewcommand{\H}{{\mathcal{H}}}

\renewcommand{\L}{{\mathcal{L}}}
\newcommand{\M}{{\mathcal{M}}}

\renewcommand{\P}{{\mathcal{P}}}

\newcommand{\X}{{\mathcal{X}}}

\renewcommand{\phi}{\varphi}

\newcommand{\qand}{\quad\text{and}\quad}
\newcommand{\qqand}{\qquad\text{and}\qquad}

\newcommand{\tr}{\mathrm{Tr}}
\newcommand{\diag}{\mathrm{diag}}
\newcommand{\op}{\mathrm{op}}
\renewcommand{\Re}{\mathrm{Re}}
\renewcommand{\Im}{\mathrm{Im}}

\usepackage{tikz}
\usetikzlibrary{shapes,snakes,calc,arrows}
\usetikzlibrary{decorations.pathreplacing,shapes.geometric}
\usetikzlibrary{calc,positioning}
\tikzset{Box/.style={very thick, rounded corners}}
\tikzset{marked/.style={star, star point height = .75mm, star points =5, fill=black,minimum size=2mm, inner sep=0mm} }
\tikzset{verythickline/.style = {line width=7pt}}
\tikzset{thickline/.style = {line width=5pt}}
\tikzset{medthick/.style = {line width=3pt}}
\tikzset{med/.style = {line width=2pt}}
\tikzset{count/.style = {fill=white,circle,draw,thin, inner sep=2pt}}
\tikzset{rcount/.style = {fill=white,rectangle,draw,thin,inner sep=2pt, rounded corners}}
\tikzset{cpr/.style = {draw,fill=white,rectangle,thin, rounded corners}}

\begin{document}

\nocite{*}

\title[Some Bi-Matrix Models for Bi-Free Limit Distributions]{Some Bi-Matrix Models for Bi-Free Limit Distributions}

\author{Paul Skoufranis}
\address{Department of Mathematics, Texas A\&M University, College Station, Texas, USA, 77843}
\email{pskoufra@math.tamu.edu}

\date{\today}

\begin{abstract}
In this paper, an analogue of matrix models from free probability is developed in the bi-free setting.  A bi-matrix model is not simply a pair of matrix models, but a pair of matrix models where one element in the pair acts by left-multiplication on matrices and the other element acts via a `twisted'-right action.  The asymptotic distributions of bi-matrix models of Gaussian random variables tend to bi-free central limit distributions with certain covariance matrices.  Furthermore, many classical random matrix results immediately generalize to the bi-free setting.  For example, bi-matrix models of left and right creation and annihilation operators on a Fock space have joint distributions equal to left and right creation and annihilation operators on a Fock space and are bi-freely independent from the left and right action of scalar matrices.  Similar results hold for bi-matrix models of $q$-deformed left and right creation and annihilation operators provided asymptotic limits are considered.  Finally bi-matrix models with asymptotic limits equal to Boolean independent central limit distributions and monotonically independent central limit distributions are constructed.
\end{abstract}

\maketitle

\section{Introduction}

Since the notion of bi-free pairs of faces was introduced by Voiculescu in \cite{V2014}, the theory has quickly developed by generalizing many ideas and results from free probability. For example, \cite{V2014} determined the bi-free central limit distributions.  On the combinatorial side, Mastnak and Nica in \cite{MN2013} introduced a collection of partitions for bi-free pairs of faces that was postulated to be analogous to the role non-crossing partitions play in free probability.  In \cite{CNS2014-1}, the postulate of Mastnak and Nica was confirmed to be correct.  Subsequently \cite{CNS2014-2} generalized such notions to the operator-valued setting where things get `interesting'.  

Following the development of free probability, Voiculescu in \cite{V2013-2} constructed a bi-free partial $R$-transform as an analogue of his $R$-transform from \cite{V1986}, which plays an important role in free probability.  Furthermore, Voiculescu in \cite{V2015} developed a bi-free partial $S$-transform and combinatorial proofs were later found in \cite{S2014} and \cite{S2015} respectively.  In addition, the notion of bi-free infinitely (additive) divisible distributions was developed in \cite{GHM2015} thereby discovering the bi-free Poisson distributions.

One important result in free probability is the connection between free probability theory and random matrix theory.  Indeed Voiculescu in \cite{V1991} connected these two theories by demonstrating that the distributions of certain matrices tend to the free central limit distributions, namely semicircular distributions, as the size of the matrices increase.  Furthermore, results pertaining to asymptotic freeness were demonstrated in \cite{V1991} and similar results for other matrix models were developed in \cites{S1995, S1997, D1993}.  The matrix models of \cites{V1991, S1995, S1997, D1993} have since been important tools in free probability.

The goal of this paper is to introduce a bi-free analogue of matrix models.  Such bi-matrix models have been elusive and the view taken in this paper is that a bi-matrix model is not simply a pair of matrices, but a pair of matrices with specific actions on matrices.  This approach of requiring actions on matrices may be slightly unsatisfactory, but the notion of random matrices acting on matrices has precedence (e.g. \cite{L2010}) and work in \cite{S2014}*{Section 6} provides evidence (beyond that presented here) for why this view is possibly the correct one.    Including this introduction, this paper contains seven sections which are summarized as follows.

The necessary background on bi-free probability is recalled in Section \ref{sec:background}.  Said background includes bi-non-crossing partitions, the bi-non-crossing M\"{o}bius function, $(\ell, r)$-cumulants, the universal moment polynomials for bi-freeness, and the abstract structures required for bi-freeness with amalgamation.  Complete expositions of these results may found in \cites{CNS2014-1, CNS2014-2}.

The structures and notion of a bi-matrix model are introduced in Section \ref{sec:structures}.  One should think of a bi-matrix model as a pair of matrices of operators that act on matrices of vectors (or operators) where the left matrix in the pair acts via left matrix multiplication whereas the right matrix in the pair acts via right matrix multiplication with a `twist'.  In addition to developing such notions, an essential lemma is presented which enables one to compute the joint distributions of bi-matrix models.

Bi-matrix models with commutative entries are examined in Section \ref{sec:commutative}.  Due to the commutation of the entries, bi-matrix models simplify slightly in this setting.  Many results analogous to those in free probability are obtained.  In particular, bi-matrix models with asymptotic distributions equal to certain bi-free central limit distributions from \cite{V2014} are demonstrated.  However, due to commutative, only bi-free central limit distributions where the left and right operators commute in distribution may be obtained.  These restriction produced by commutativity have already been seen in the theory of bi-free pairs of faces in that the operator model in \cite{MN2013} using left and right creation operators on a Fock space was insufficient to describe all bi-free distributions and a more complicated model was required (see \cite{CNS2014-1}). 

 However, commutativity is pleasant in bi-free probability as it enables the importation of results from free probability.   In fact, Theorem \ref{thm:ass-free-to-ass-bi-free} demonstrates that asymptotic bi-freeness for commutative bi-matrix models is not much more than asymptotic freeness of random matrices.   Consequently most (if not all) results relating random matrix theory to free probability should have an analogue in the bi-free setting.  Therefore we only exhibit specific results such as realizing bi-free Poisson distributions via random pairs of Wishart matrices and demonstrating the asymptotic bi-freeness from constant matrices of random pairs of Gaussian (or Haar) unitary matrices.

Some inadequacies from Section \ref{sec:commutative} are rectified in Section \ref{sec:fock}.  In particular, generalizing \cite{S1997}*{Theorem 5.2},  bi-matrix models where the left matrices consists of left creation and annihilation operators on a Fock space and the right matrices consists of right creation and annihilation operators have joint distributions are equal to the joint distribution of left and right creation and annihilation operators.  Using these models, all bi-free central limit distributions may be obtained.  Furthermore, such bi-matrix models are bi-free from the left and right actions of constant matrices.

Bi-matrix models involving $q$-deformed left and right creation and annihilation operators are examined in Section \ref{sec:q-deformed}.  In particular, asymptotic  bi-freeness from block matrices is obtain and the asymptotic limit of these bi-matrix models is again left and right creation and annihilation operators on a Fock space.  The results of this section generalize those of \cite{S1995} to the bi-free setting and provide an alternate approach to Theorem \ref{thm:ass-bi-free} by the $q = 1$ case.

Finally, in Section \ref{sec:boolean-monotone}, bi-matrix models with asymptotic limits equal to Boolean independent central limit distributions and monotonically independent central limit distributions are constructed.   These results require the use of products of left and right matrices to represent a single random variable in order to obtain the desired limit distributions.

\section{Background on Bi-Free Pairs of Faces}
\label{sec:background}

In this section, we will summarize some essential combinatorial aspects of bi-free probability from \cite{CNS2014-1} and the structures for operator-valued bi-free probability from \cite{CNS2014-2}.  We will only present the results essential to this paper and refer the reader to complete summarizes in \cite{CNS2014-2} and \cite{S2014} respectively.

\subsection*{Bi-Freeness for Pairs of Faces}
For the remainder of this paper, a map $\chi : \{1,\ldots, n\} \to \{\ell, r\}$ is used to designate whether the $k^{\mathrm{th}}$ operator in a sequence of $n$ operators should be a left operator (when $\chi(k) = \ell$) or a right operator (when $\chi(k) = r$).  Similarly, a map $\epsilon : \{1,\ldots, n\} \to K$ is used to determine which colour from a collection $K$ each operator is coloured.

The main difference between the combinatorial aspects of free and bi-free probability stem from handling the following permutation.
\begin{defn}
Given $\chi: \{1, \ldots, n\} \to \{\ell, r\}$, if 
\[
\chi^{-1}(\{\ell\}) = \{i_1<\cdots<i_p\} \qqand \chi^{-1}(\{r\}) = \{i_{p+1} > \cdots > i_n\},
\]
the permutation $s_\chi$ on $\{1,\ldots, n\}$ is defined by $s_\chi(k) = i_k$.
\end{defn}
Note the sequence $(s_\chi(1), \ldots, s_\chi(n))$ corresponds to, instead of reading $\{1,\ldots, n\}$ in the traditional order, reading $\chi^{-1}(\{\ell\})$ (that is, the left nodes) in increasing order followed by reading $\chi^{-1}(\{r\})$ (that is, the right nodes) in decreasing order.  

Let $\P (n)$ denote the set of partitions on $n$ elements.  Given a partition $\pi \in \P (n)$ and $x, y \in \{1,\ldots, n\}$, we write $x \sim_\pi y$ whenever $x$ and $y$ are in the same block of $\pi$ (a set in $\pi$) and $x \nsim_\pi y$ otherwise.  Given two partitions $\pi, \sigma \in \P (n)$, $\pi$ is said to be a refinement of $\sigma$, denoted $\pi \leq \sigma$, if every block of $\pi$ is contained in a single block of $\sigma$.  Refinement defines a partial ordering on $\P (n)$ turning $\P (n)$ into a lattice.

\begin{defn}
A partition $\pi \in \P (n)$ is said to be \emph{bi-non-crossing with respect to $\chi$} if the partition  $s_\chi^{-1}\cdot \pi$ (the partition formed by applying $s_\chi^{-1}$ to the blocks of $\pi$) is non-crossing.
Equivalently $\pi$ is bi-non-crossing if $\pi$ is non-crossing on $\{1,\ldots, n\}$ when the set is ordered via $s_\chi(1) \prec_\chi s_\chi(2) \prec_\chi \cdots \prec_\chi s_\chi(n)$.  The set of bi-non-crossing partitions with respect to $\chi$ is denoted by $BNC(\chi)$.
\end{defn}

\begin{exam}
If $\chi : \{1,\ldots, 6\} \to \{\ell, r\}$ is such that $\chi^{-1}(\{\ell\}) = \{2, 3, 6\}$ and $\chi^{-1}(\{r\}) = \{1, 4, 5\}$, then $(s_\chi(1), \ldots, s_\chi(6)) = (2, 3, 6, 5, 4, 1)$.  If $\pi = \{ \{1,4\}, \{2,5\}, \{3, 6\}\}$, then $\pi$ is crossing on $\{1,\ldots, 6\}$ yet is bi-non-crossing with respect to $\chi$.  This may be seen via the following diagrams.
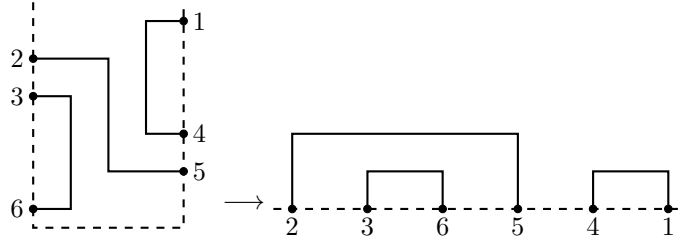
\begin{align*}
\begin{tikzpicture}[baseline]
	\draw[thick, dashed] (-1,2.75) -- (-1,-.25) -- (1,-.25) -- (1,2.75);
	\draw[thick] (-1, 2) -- (0,2) -- (0,.5) -- (1, .5);
	\draw[thick] (-1, 1.5) -- (-.5,1.5) -- (-0.5,0) -- (-1, 0);
	\draw[thick] (1, 2.5) -- (0.5,2.5) -- (0.5,1) -- (1, 1);
	\node[right] at (1, 2.5) {$1$};
	\draw[black, fill=black] (1,2.5) circle (0.05);	
	\node[left] at (-1, 2) {$2$};
	\draw[black, fill=black] (-1,2) circle (0.05);
	\node[left] at (-1, 1.5) {$3$};
	\draw[black, fill=black] (-1,1.5) circle (0.05);
	\node[right] at (1, 1) {$4$};
	\draw[black, fill=black] (1,1) circle (0.05);
	\node[right] at (1, .5) {$5$};
	\draw[black, fill=black] (1,.5) circle (0.05);
	\node[left] at (-1, 0) {$6$};
	\draw[black, fill=black] (-1,0) circle (0.05);
\end{tikzpicture}
\longrightarrow
\begin{tikzpicture}[baseline]
	\draw[thick, dashed] (.75,0) -- (6.25, 0);
	\draw[thick] (2, 0) -- (2,.5) -- (3,.5) -- (3, 0);
	\draw[thick] (5, 0) -- (5,.5) -- (6,.5) -- (6, 0);
	\draw[thick] (1, 0) -- (1,1) -- (4,1) -- (4, 0);
	\node[below] at (1, 0) {$2$};
	\draw[black, fill=black] (1,0) circle (0.05);	
	\node[below] at (2, 0) {$3$};
	\draw[black, fill=black] (2,0) circle (0.05);	
	\node[below] at (3, 0) {$6$};
	\draw[black, fill=black] (3,0) circle (0.05);	
	\node[below] at (4, 0) {$5$};
	\draw[black, fill=black] (4,0) circle (0.05);	
	\node[below] at (5, 0) {$4$};
	\draw[black, fill=black] (5,0) circle (0.05);	
	\node[below] at (6, 0) {$1$};
	\draw[black, fill=black] (6,0) circle (0.05);	
\end{tikzpicture}
\end{align*}
\end{exam}

One function on pairs of elements of $BNC(\chi)$ required in this paper is the following.
\begin{defn}
\label{defn:mobius}
The \emph{bi-non-crossing M\"{o}bius function} is the function
\[
\mu_{BNC} : \bigcup_{n\geq1}\bigcup_{\chi : \{1, \ldots, n\}\to\{\ell, r\}}BNC(\chi)\times BNC(\chi) \to \bC
\]
defined such that $\mu_{BNC}(\pi, \sigma) = 0$ unless $\pi$ is a refinement of $\sigma$, and otherwise defined recursively via the formulae
\[
\sum_{\substack{\tau \in BNC(\chi) \\\pi \leq \tau \leq \sigma}} \mu_{BNC}(\tau, \sigma) = \sum_{\substack{\tau \in BNC(\chi) \\ \pi \leq \tau \leq \sigma}} \mu_{BNC}(\pi, \tau) = \left\{
\begin{array}{ll}
1 & \mbox{if } \pi = \sigma  \\
0 & \mbox{otherwise}
\end{array} \right. .
\]
\end{defn}

Due to the similarity in lattice structures, the bi-non-crossing M\"{o}bius function is related to the non-crossing M\"{o}bius function $\mu_{NC}$ by the formula 
\[
\mu_{BNC}(\pi, \sigma) = \mu_{NC}(s^{-1}_\chi \cdot \pi, s^{-1}_\chi \cdot \sigma).
\]
This implies that $\mu_{BNC}$ inherits many `multiplicative' properties that $\mu_{NC}$ has.  For more details, see \cite{CNS2014-1}*{Section 3}.

\begin{defn}
Let $(\A, \varphi)$ be a non-commutative probability space: that is, let $\A$ be a unital algebra and $\varphi : \A \to \mathbb{C}$ be unital and linear.  A \emph{pair of faces in $\A$} is a pair $(C, D)$ of unital subalgebras of $\A$.  We will call $C$ the \emph{left face} and $D$ the \emph{right face}.
\end{defn}

Given $Z_1, \ldots, Z_n \in \A$ and $\pi \in \P(n)$ with blocks $V_x = \{k_{1,x} < \cdots < k_{m_x, x}\}$ for $x \in \{1, \ldots, p\}$, define
\[
\varphi_\pi(Z_1, \ldots, Z_n) := \prod^p_{x=1} \varphi(Z_{k_{1, x}} \cdots Z_{k_{m_x, x}}).
\]
Given a map $\chi : \{1,\ldots, n\} \to \{\ell, r\}$ and $\pi \in BNC(\chi)$, define
\begin{align}
\kappa_\pi(Z_1, \ldots, Z_n) := \sum_{\substack{\sigma \in BNC(\chi) \\ \sigma \leq\pi}} \varphi_\sigma(Z_1, \ldots, Z_n) \mu_{BNC}(\sigma, \pi).  \label{eq:cum-in-terms-of-mom}
\end{align}
The \emph{$(\ell, r)$-cumulant of $Z_1, \ldots, Z_n$ corresponding to $\chi$} is $\kappa_\chi(Z_1, \ldots, Z_n) := \kappa_{1_\chi}(Z_1, \ldots, Z_n)$
where $1_\chi$ is the full partition $\{\{1,\ldots, n\}\}$.  Each $(\ell, r)$-cumulant should always be viewed as a function where only elements of left faces may be inserted into the $k^{\mathrm{th}}$ entry when $\chi(k) = \ell$ and only elements of right faces may be inserted into the $k^{\mathrm{th}}$ entry when $\chi(k) = r$.  Furthermore, if $\pi \in BNC(\chi)$ has blocks $V_x = \{k_{1,x} < \cdots < k_{m_x, x}\}$ for $x \in \{1, \ldots, p\}$, then one can verify 
\[
\kappa_\pi(Z_1, \ldots, Z_n) = \prod^p_{x=1} \kappa_{\pi|_{V_x}}(Z_{k_{1, x}}, \ldots,  Z_{k_{m_x, x}}).
\]
Furthermore, the reversion formula
\begin{align}
\varphi(Z_1\cdots Z_n) = \sum_{\pi \in BNC(\chi)} \kappa_\pi(Z_1, \ldots, Z_n) \label{eq:mom-in-terms-of-cum}
\end{align}
holds thereby giving each moment as a sum of products of $(\ell, r)$-cumulants.

The following is the main result from \cite{CNS2014-1}.  For $\epsilon : \{1,\ldots, n\} \to K$, note $\epsilon$ defines an element of $\P(n)$ whose blocks are $\{\epsilon^{-1}(\{k\})\}_{k \in K}$.
\begin{thm}[\cite{CNS2014-1}*{Theorem 4.3.1}]
\label{thm:universal-moment-bi-free}
Let $(\A, \varphi)$ be a non-commutative probability space and let $\{(C_k, D_k)\}_{k \in K}$ be a family of pairs of faces from $\A$.  Then $\{(C_k, D_k)\}_{k \in K}$ are bi-free with respect to $\varphi$ if and only if for all $\chi : \{1,\ldots, n\}\to\{\ell,r\}$, $\epsilon : \{1,\ldots, n\} \to K$, and 
\[
Z_m \in \left\{
\begin{array}{ll}
C_{\epsilon(m)} & \mbox{if } \chi(m) = \ell  \\
D_{\epsilon(m)} & \mbox{if } \chi(m) = r
\end{array} \right.,
\]
we have 
\begin{align}
\varphi(Z_1 \cdots Z_n) = \sum_{\pi \in BNC(\chi)} \left[ \sum_{\substack{\sigma \in BNC(\chi) \\ \pi \leq \sigma \leq \epsilon}} \mu_{BNC}(\pi, \sigma) \right]\varphi_\pi(Z_1,\ldots, Z_n). \label{eq:universal-polys}
\end{align}
Furthermore, it suffices to verify equation (\ref{eq:universal-polys}) for $Z_k$ in generating subsets of $C_{\epsilon(k)}$ and $D_{\epsilon(k)}$ and bi-freeness of $\{(C_k, D_k)\}_{k \in K}$ is equivalent to $\kappa_\chi(Z_1, \ldots, Z_n) = 0$ whenever $\chi$, $\epsilon$, and $Z_k$ are as above and $\epsilon$ is not constant.
\end{thm}

Important to this paper are the bi-free central limit distributions of \cite{V2014}*{Section 7}.
\begin{defn}
Let $(\A,\varphi)$ be a non-commutative probability space and let $I$ and $J$ be disjoint index sets. A two-faced family $(\{Z_i\}_{i \in I}, \{Z_j\}_{j \in J})$ is said to be a \emph{(centred) bi-free central limit distribution} if  all $(\ell, r)$-cumulants of order 1 and of order at least 3 are zero.
\end{defn}
Note \cite{V2014}*{Theorem 7.4} completely describes all possible (centred) bi-free central limit distribution via the (complex) numbers $\varphi(Z_{k_1} Z_{k_2}) = \kappa_\chi(Z_{k_1}, Z_{k_2})$ for $k_1, k_2 \in I \sqcup J$, and concretely represents such two-faced families as combinations of left and right creation and annihilation operators on Fock spaces.   The matrix $C$ indexed by $I \sqcup J$ obtained with the $(k_1,k_2)$ entry equal to $\varphi(Z_{k_1} Z_{k_2})$ is called \emph{the covariance matrix of the family $(\{Z_i\}_{i \in I}, \{Z_j\}_{j \in J})$}.

\subsection*{Structures for Operator-Valued Bi-Freeness}

With \cite{CNS2014-2} things just got stranger as there are two copies of the amalgamation algebra and the natural structures for operator-valued bi-free probability appear, on the surface, very different from from those for operator-valued free probability.  We describe these structures here.  For the remainder of this section $B$ will denote a unital algebra.  One may replace $B$ with the $N \times N$ matrix algebra $\M_N(\bC)$ as this is all that is required in other sections of this paper.
\begin{defn}
A \emph{$B$-non-commutative probability space} is a pair $(\A, \Phi)$ where $\A$ is a unital algebra containing $B$ (with $1_\A= 1_B$) and $\Phi : \A \to B$ is a unital linear map such that
\[
\Phi(b_1 Z b_2) = b_1 \Phi(Z) b_2
\]
for all $b_1, b_2 \in B$ and $Z \in \A$.
\end{defn}

\begin{defn}
A \emph{$B$-$B$-non-commutative probability space} is a triple $(\A, E_\A, \varepsilon)$ where $\A$ is a unital algebra, $\varepsilon : B \otimes B^{\op} \to \A$ is a unital homomorphism such that $\varepsilon|_{B \otimes 1_B}$ and $\varepsilon|_{1_B \otimes B^{\op}}$ are injective, and $E_\A : \A \to B$ is a linear map such that
\[
E_{\A}(\varepsilon(b_1 \otimes b_2)Z) = b_1 E_{\A}(Z) b_2
\qqand
E_{\A}(Z\varepsilon(b \otimes 1_B)) = E_{\A}(Z\varepsilon(1_B \otimes b))
\]
for all $b_1, b_2, b \in B$ and $Z \in \A$.  To simplify notation, $L_b$ and $R_b$ are used in place of $\varepsilon(b \otimes 1_B)$ and $\varepsilon(1_B \otimes b)$ respectively. 

The unital subalgebras of $\A$ defined by
\begin{align*}
\A_\ell &:= \{ Z \in \A  \, \mid \, Z R_b = R_b Z \mbox{ for all }b \in B\} \text{ and}\\
\A_r &:= \{ Z \in \A  \, \mid \, Z L_b = L_b Z \mbox{ for all }b \in B\}
\end{align*}
are called the \emph{left} and \emph{right algebras of} $\A$ respectively.  
\end{defn}

Given a $B$-$B$-non-commutative probability space $(\A, E_\A, \varepsilon)$, notice $(\A_\ell, E_\A)$ is always a $B$-non-commutative probability space with $\varepsilon(B \otimes 1_B)$ as the copy of $B$ and $(\A_r, E_\A)$ is a $B^{\op}$-non-commutative probability space with $\varepsilon(1_B \otimes B^{\op})$ as the copy of $B^{\op}$.  Furthermore $(\A, E_\A, \varepsilon)$ is simply a non-commutative probability space whenever $B = \bC$.

The reason $B$-$B$-non-commutative probability spaces are the correct setting of operator-valued bi-free probability can be seen through the following.

\begin{defn}
A \emph{$B$-$B$-bimodule with a specified $B$-vector state} is a triple $(\X, \mathring{\X}, p_\X)$ where $\X$ is a direct sum of $B$-$B$-bimodules
\[
\X = B \oplus \mathring{\X},
\]
and $p_\X : \X \to B$ is the linear map
\[
p_\X(b \oplus \eta) = b.
\]

Let $\L(\X)$ denote the set of linear operators on $\X$.  For each $b \in B$ define the operators $L_b, R_b \in \L(\X)$ by
\[
L_b(\eta) = b \cdot \eta \qquad \mbox{ and } \qquad R_b(\eta) = \eta \cdot b \qquad \text{ for all } \eta \in \X.
\]
The unital subalgebras of $\L(\X)$ defined by
\begin{align*}
\L_\ell(\X) &:= \{ Z \in \L(\X) \, \mid \, ZR_b = R_b Z \mbox{ for all }b \in B\} \text{ and}\\
\L_r(\X) &:= \{ Z \in \L(\X) \, \mid \, ZL_b = L_b Z \mbox{ for all }b \in B\}
\end{align*}
are called the \emph{left} and \emph{right algebras of} $\L(\X)$ respectively.

Given a $B$-$B$-bimodule with a specified $B$-vector state $(\X, \mathring{\X}, p)$, the \emph{expectation of $\L(\X)$ onto $B$} is the linear map $E_{\L(\X)} : \L(\X) \to B$ defined by
\[
E_{\L(\X)}(Z) = p_\X(Z(1_B))
\]
for all $Z \in \L(\X)$.  
\end{defn}
Notice if $\varepsilon : B \otimes B^{\op} \to \L(\X)$ is defined by $\varepsilon(b_1 \otimes b_2) = L_{b_1} R_{b_2}$, then $(\L(\X), E_{\L(\X)}, \varepsilon)$ is a concrete $B$-$B$-non-commutative probability space.  Moreover \cite{CNS2014-2}*{Theorem 3.2.4} demonstrated every abstract $B$-$B$-non-commutative probability space can be represented inside a concrete $B$-$B$-non-commutative probability space, at least for the purposes of computing distributions.

\section{General Structure for Bi-Matrix Models}
\label{sec:structures}

The general structures required for bi-matrix models are introduced in this section.  These structures are motivated by the main result of \cite{S2014}*{Section 6} which is stated below.  The symbols $\{E_{i,j}(N)\}^N_{i,j=1}$ will denote the standard matrix units for $\M_N(\bC)$.  Either $[a_{i,j}]$ or $\sum^N_{i,j=1} a_{i,j} \otimes E_{i,j}(N)$ is used to denote the $N \times N$ matrix with $(i,j)^{\mathrm{th}}$ entry $a_{i,j}$ depending on which better fits the current context.  Furthermore $\tr$ will denote the trace on $\M_N(\bC)$ defined by
\[
\tr([a_{i,j}]) = \sum^N_{i=1} a_{i,i}.
\]

The following introduces the general structure for bi-matrix models for matrices with elements in $\L(\X)$ where $\X$ is a pointed vector space.
\begin{cons}
\label{cons:twisting-action}
Let $(\X, \mathring{\X}, \xi, p_\X)$ be a pointed vector space; that is, $\X$ is a vector space over $\bC$ with $\X = \bC \xi \oplus \mathring{\X}$ and $p_\X : \X \to \bC$ is the linear map defined by 
\[
p_\X(\lambda \xi \oplus \eta) = \lambda.
\]

For $N \in \bN$ consider the $\M_N(\bC)$-$\M_N(\bC)$-bimodule $\X_N := \M_N(\X)$ where
\[
[a_{i,j}] \cdot [\eta_{i,j}] = \left[  \sum^N_{k=1} a_{i,k} \eta_{k,j}\right] \qqand [\eta_{i,j}] \cdot [a_{i,j}] = \left[  \sum^N_{k=1} a_{k,j} \eta_{i,k}\right]
\]
for all $[a_{i,j}] \in \M_N(\bC)$ and $[\eta_{i,j}] \in \X_N$.  Then $\X_N$ becomes an $\M_N(\bC)$-$\M_N(\bC)$-bimodule with specified $\M_N(\bC)$-vector state via
\[
\X_N = \M_N(\bC \xi) \oplus \M_N(\mathring{\X}),
\]
and the linear map $p_{\X_N} : \X_N \to \M_N(\bC)$ defined by
\[
p_{\X_N}([\eta_{i,j}]) = [p_\X(\eta_{i,j})].
\]
We call $\X_N$ \emph{the $\M_N(\bC)$-$\M_N(\bC)$-bimodule associated with $(\X, p_\X)$} and $(\L(\X_N), E_{\L(\X_N)})$ \emph{the $\M_N(\bC)$-$\M_N(\bC)$-non-commutative probability space associated with $(\X, p_\X)$}.  

Since $\X$ is a pointed vector space, there is a canonical choice of linear functional $\varphi_\X : \L(\X) \to \bC$ defined by 
\[
\varphi_\X(Z) = p_\X( Z(\xi)).
\]
Therefore, the expectation $E_{\L(\X_N)} : \L(\X_N) \to \M_N(\bC)$ is defined by
\[
E_{\L(\X_N)}(Z) = p_{\X_N}(Z I_{N,\xi})
\]
where $I_{N,\xi}$ is the diagonal matrix $\diag(\xi, \xi, \ldots, \xi)$.

To consider bi-matrix models, two homomorphisms from $\M_N(\L(\X))$ into $\L(\X_N)$ are required.   Define $L : \M_N(\L(\X)) \to \L(\X_N)$ by
\[
L([T_{i,j}]) [\eta_{i,j}] = \left[ \sum^N_{k=1} T_{i,k}(\eta_{k,j})\right]
\]
for all $[\eta_{i,j}] \in \X_N$ and $[T_{i,j}] \in \M_N(\L(\X))$.  Is is elementary to verify that $L$ is a homomorphism, that $L(\M_N(\L(\X))) \subseteq \L_\ell(\X_N)$, and that $L([a_{i,j} I_\X]) = L_{[a_{i,j}]}$.  On the other hand (or perhaps, on the other face), define $R : \M_N(\L(\X)^{\op})^{\op} \to \L(\X_N)$ by
\[
R([S_{i,j}]) [\eta_{i,j}] = \left[ \sum^N_{k=1} S_{k,j} (\eta_{i,k})\right]
\]
for all $[\eta_{i,j}] \in \X_N$ and $[S_{i,j}] \in \M_N(\L(\X))$.  Is is elementary to verify that $R$ is a homomorphism (hence the ops), that $R(\M_N(\L(\X)^\op)^\op) \subseteq \L_r(\X_N)$, and that $R([a_{i,j} I_\X]) = R_{[a_{i,j}]}$.  We use $L$ and $R$ instead of $L_N$ and $R_N$ as the size of the input matrices determine $N$.  The image of $L$ (respectively $R$) is called \emph{the left (respectively right) matrix algebras of $\L(\X)$} and elements of this algebra are called \emph{left (respectively right) matrices of $\L(\X)$}.

\end{cons}
\begin{rem}
\label{rem:commutative}
If $[T_{i,j}], [S_{i,j}] \in \M_N(\L(\X))$ are such that $T_{i,j} S_{k,m} = S_{k,m} T_{i,j}$ for all $i,j,k,m$, then it is elementary to verify that  $L([T_{i,j}]) R([S_{i,j}]) = R([S_{i,j}]) L([T_{i,j}])$.
\end{rem}

\begin{rem}
\label{rem:transitioning-from-left-to-right}
It is worth mentioning that for all $[Z_{i,j}] \in \M_N(\L(\X))$,
\[
L([Z_{i,j}]) I_{N, \xi} = R([Z_{i,j}]) I_{N, \xi}.
\]
This fact will be important later in the paper in order to `transition from right operators to left operators'; a theme that has been essential in bi-free probability (see \cites{CNS2014-1, CNS2014-2, GHM2015}).

\begin{rem}
\label{rem:distributions-of-matrices-depend-only-on-original-state}
Using the notation as in Construction \ref{cons:twisting-action}, if $[Z_{i,j}] \in \M_N(\L(\X))$ then
\[
E_{\L(\X_N)}(L([Z_{i,j}])) = p_{\X_N}(L([Z_{i,j}])I_{N,\xi}) = p_{\X_N}([Z_{i,j} (\xi)]) = [p_\X(Z_{i,j}(\xi))] = [\varphi_\X(Z_{i,j})] = E_{\L(\X_N)}(R([Z_{i,j}])).
\]
As similar computations hold for any product of left and right matrices of $\L(\X)$, we see that $E_{\L(\X_N)}$ applied to a product of left and right matrices of $\L(\X)$ depends only on the sequence of left and right matrices, the entries in the left and right matrices, and the linear map $\varphi_\X : \L(\X) \to \bC$.  In particular, for such products, $E_{\L(\X_N)}$ is the expected expectation obtained by applying $\varphi_\X$ to each entry of an element of $\M_N(\L(\X))$.
\end{rem}

\end{rem}
\begin{rem}
\label{rem:non-commutative-probability-space-canonical-action}
Given a non-commutative probability space $(\A, \varphi)$, there are many ways to view $\A \subseteq \L(\X)$ in a state-preserving way.  In particular, we adopt the convention that $\X = \A$ with $p_\X(A) = \varphi(A)$ (so $\X = \bC \oplus \ker(\varphi)$) and $\A \subseteq \L(\X)$ via $T(A) = TA$ for all $A \in \X$ and $T \in \A \subseteq \L(\X)$.  In particular, this convention implies 
\[
\varphi_\X(T) = p_\X(T (I_\A)) = p_\X(T) = \varphi(T).
\]
Under this convention and using the notation of Construction \ref{cons:twisting-action}, we see for $N \in \bN$ that $\X_N = \M_N(\A)$ and
\[
L([Z_{i,j}]) [A_{i,j}] = \left[ \sum^N_{k=1} Z_{i,k}A_{k,j}\right] \qand R([Z_{i,j}]) [A_{i,j}] = \left[ \sum^N_{k=1} Z_{k,j} A_{i,k}\right]
\]
for all $[Z_{i,j}], [A_{i,j}] \in \M_N(\A)$.  This computation shows why one might call $R$ a `twisted'-right action as it looks like multiplying the matrices $[A_{i,j}][Z_{i,j}]$ in $\M_N(\A)$ except that the opposite multiplication on $\A$ is used in each matrix entry.

Using the same idea as in Remark \ref{rem:distributions-of-matrices-depend-only-on-original-state}, one sees that the joint distribution of $L(\M_N(\A))$ and $R(\M_N(\A^{\op})^{\op})$ depends only on the sequence of left and right matrices, the entries in the left and right matrices, and the linear map $\varphi : \A \to \bC$. 

It is important to emphasize one does not want the operators $Z_{i,j}$ in $R([Z_{i,j}])$ to be elements of $\A$ acting `on the right' of $\A$; that is, via the action $S(A) = AS$.  One reason for this, other than Theorem \ref{thm:bi-free-with-amalgamation-over-matrix-algebra}, is that $L([T_{i,j}])$ and $R([S_{i,j}])$ would commute by Remark \ref{rem:commutative} and too much commutativity in bi-free probability is not optimal.  
\end{rem}

Previous evidence towards why Construction \ref{cons:twisting-action} is the correct mathematical construction to use is the following result which demonstrates how bi-freeness of pairs or faces extends to bi-freeness of matrices of the pairs of faces.
\begin{thm}[specific case of \cite{S2014}*{Theorem 6.3.1}]
\label{thm:bi-free-with-amalgamation-over-matrix-algebra}
Let $(\A, \varphi)$ be a non-commutative probability space and let $\{(C_k, D_k)\}_{k \in K}$ be bi-free pairs of faces with respect to $\varphi$.  For each $N \in \bN$ let $(\L(\X_N), E_{\L(\X_N)})$ be the canonical $\M_N(\bC)$-$\M_N(\bC)$-non-commutative probability space associated with $\A$ as in Remark \ref{rem:non-commutative-probability-space-canonical-action}.  Then $\{(L(\M_N(C_k)), R(\M_N(D_k^{\op})^\op))\}_{k \in K}$ are bi-free with amalgamation over $\M_N(\bC)$ with respect to $E_{\L(\X_N)}$.
\end{thm}

A piece of the proof of Theorem \ref{thm:bi-free-with-amalgamation-over-matrix-algebra} essential to this paper is the following computation.
\begin{lem}
\label{lem:expanding-matrices}
Using the notation and conventions of Construction \ref{cons:twisting-action}, let $\chi : \{1, \ldots, n\} \to \{\ell, r\}$, and let $Z_k = L([Z_{i,j; k}])$ if $\chi(k) = \ell$ and $Z_k = R([Z_{i,j; k}])$ if $\chi(k) = r$.  Then
\[
(Z_1 \cdots Z_n)(I_{N, \xi}) = \sum^N_{\substack{i_1, \ldots, i_n = 1 \\ j_1, \ldots, j_n = 1}} ((Z_{i_1, j_1; 1} \circ \cdots\circ Z_{i_n, j_n; n})(\xi))  \otimes E_\chi((i_1, \ldots, i_n), (j_1, \ldots, j_n); N)
\]
where
\[
E_\chi((i_1, \ldots, i_n), (j_1, \ldots, j_n); N) = E_{i_{s_\chi(1)}, j_{s_\chi(1)}}(N) \cdots E_{i_{s_\chi(n)}, j_{s_\chi(n)}}(N) \in \M_N(\bC).
\]
\end{lem}
\begin{proof}
By the linearity of all maps involved, it suffices to consider $Z_k = L(Z_{i_k, j_k} \otimes E_{i_k, j_k}(N))$ when $\chi(k) = \ell$ and $Z_k = R(Z_{i_k, j_k} \otimes E_{i_k, j_k}(N))$ when $\chi(k) = r$.  Note
\[
Z_n(I_{N, \xi}) = Z_{i_n, j_n} (\xi) \otimes E_{i_n, j_n}(N)
\]
regardless of the value of $\chi(n)$.  

If $\chi(n-1) = \ell$, observe that
\[
Z_{n-1} Z_n(I_{N, \xi}) = (Z_{i_{n-1}, j_{n-1}} Z_{i_n, j_n} (\xi)) \otimes E_{i_{n-1}, j_{n-1}}(N)E_{i_n, j_n}(N)
\]
whereas, if $\chi(n-1) = r$, observe that
\[
Z_{n-1} Z_n(I_{N, \xi}) = (Z_{i_{n-1}, j_{n-1}} Z_{i_n, j_n} (\xi)) \otimes E_{i_n, j_n}(N)E_{i_{n-1}, j_{n-1}}(N).
\]
In particular, this pattern repeats where $E_{i_k, j_k}(N)$ is placed on the left-hand side of the product of matrix units if $\chi(k) = \ell$ whereas $E_{i_k, j_k}(N)$ is placed on the right-hand side of the product of matrix units if $\chi(k) = r$.  Hence the result follows by the definition of $s_\chi$.
\end{proof}
\begin{rem}
\label{rem:indices-that-matter}
For $E_\chi((i_1, \ldots, i_n), (j_1, \ldots, j_n); N)$ to be non-zero and to be a diagonal entry (thereby contributing to the trace), it is required that $j_{s_\chi(k)} = i_{s_\chi(k+1)}$ for all $k$ (where $n+1 \to 1$).  This implies $j_k = j_{s_\chi(s^{-1}_\chi(k))} = i_{s_\chi(s^{-1}_\chi(k) + 1)}$ for all $k$.  
\end{rem}

\section{The Commutative Case}
\label{sec:commutative}

In this section, we will study bi-matrix models in the commutative case to obtain (commutative in distribution) bi-free central limit distributions.  In particular, for this section we take $\A = (L_\infty(\Omega, \mu), E)$ for our non-commutative (well, commutative) probability space where
\[
E(f) = \int_\Omega f(x) \, d\mu.
\]

Although some results in this section may be obtained via the bi-matrix models with bosonic creation and annihilation in Section \ref{sec:q-deformed} (i.e. with $q = 1$), the ideas in this section are not only simpler but important for demonstrating that many random matrix models have direct analogues in the bi-free setting.  In particular, this section can be summarized as, ``there is a bi-free analogue of any random matrix model that can be generalized from random matrices with independent entries to pairs of random matrices where each pair is independent from each other pair and each pair has a certain covariance matrix associated to it."

We begin by recalling some definitions.
\begin{defn}
\label{defn:Gaussian}
A family $X_1, \ldots, X_n$ of self-adjoint random variables in $(L_\infty(\Omega, \mu), E)$ is a \emph{(centred) Gaussian family} if there exists a non-singular positive $n\times n$ matrix $C$ with real entries (called the \emph{covariance matrix}) such that for all $k \in \bN$ and all $1 \leq i_1, \ldots, i_k \leq n$,
\[
E(X_{i_1} \cdots X_{i_k}) = \frac{1}{\sqrt{(2\pi)^n \det(C)}} \int_{\bR^n} x_{i_1} \cdots x_{i_n} e^{-\frac{1}{2} \langle C^{-1} \vec{x}, \vec{x}\rangle} \, dx_1 \cdots dx_n
\]
where $\vec{x} = (x_1, \ldots, x_n)$ and $\langle \, \cdot, \cdot \, \rangle$ denotes the standard inner product on $\bR^n$.

A family of complex random variables $X_1, \ldots, X_n$ in a $*$-probability space $(L_\infty(\Omega, \mu), E)$ is a \emph{complex Gaussian family} if $\Re(X_1), \ldots, \Re(X_n), \Im(X_1), \ldots, \Im(X_n)$ is a Gaussian family.
\end{defn}

Of important use in this section is the following formula.
\begin{thm}[Wick's Formula; see \cite{NS2006}*{Theorem 22.3} for example]
\label{thm:Wick}
Let $X_1, \ldots, X_n$ be a Gaussian family in $(L_\infty(\Omega, \mu), E)$ with covariance matrix $C = [c_{i,j}]$.  For all $k \in \bN$ and $1 \leq k_1, \ldots, k_n \leq n$, 
\[
E(X_{k_1} \cdots X_{k_n}) = \sum_{\pi \in \P_2(n)} \prod_{\{x,y\} \in \pi} E(X_{k_x} X_{k_y})
\]
where $\P_2(n)$ denotes all pair partitions on $\{1,\ldots, n\}$.  Furthermore 
\[
E(X_i X_j) = c_{i,j}.
\]
\end{thm}

\begin{exam}
\label{exam:diag}
If
\[
C = \left[  \begin{array}{cc} c_{X,X} & c_{X,Y}\\ c_{Y,X} & c_{Y,Y}  \end{array} \right]
\]
is a non-singular positive matrix with real entries, then, by Definition \ref{defn:Gaussian} and Theorem \ref{thm:Wick}, there exists a  Gaussian family $X, Y$ such that
\[
E(X^2) = c_{X,X}, \quad E(XY) = c_{X, Y}, \quad E(YX) = c_{Y, X}, \qand E(Y^2) = c_{Y, Y}.
\]
\end{exam}

\begin{exam}
\label{diag:off-diag}
If
\[
C = \left[  \begin{array}{cc} c_{X,X} & c_{X,Y}\\ c_{Y,X} & c_{Y,Y}  \end{array} \right]
\]
is a non-singular positive matrix with real entries, then we claim there exists a complex Gaussian family $X, Y$ such that
\[
E(X^2) = 0, \quad E(XY) = 0, \quad E(YX) = 0, \qand E(Y^2) = 0
\]
yet
\[
E\left(X \overline{X}\right) = c_{X,X}, \quad E\left(X\overline{Y}\right) = c_{X, Y}, \quad E\left(Y\overline{X}\right) = c_{Y, X}, \qand E\left(Y\overline{Y}\right) = c_{Y, Y}.
\]
Indeed consider the Gaussian family $Z_1, \ldots, Z_4$ with the covariance matrix
\[
\frac{1}{2}\left(C \otimes I_2\right) = \frac{1}{2} \left[  \begin{array}{cccc} c_{X,X}  & 0 & c_{X, Y} & 0 \\ 0 & c_{X,X} & 0 & c_{X, Y} \\ c_{Y, X} & 0 & c_{Y, Y} & 0 \\ 0 & c_{Y, X} & 0 & c_{Y, Y} \end{array} \right],
\]
which is clearly a non-singular positive matrix with real entries.  Therefore, if $X = Z_1 + iZ_2$ and $Y = Z_3 + iZ_4$, then $X, Y$ is a complex Gaussian family that satisfies the above equations.  Indeed
\begin{align*}
E\left(X^2\right) &= E\left(Z_1^2\right) + iE(Z_1Z_2) + i E(Z_2Z_1) - E\left(Z_2^2\right)  = 0,\\
E\left(X \overline{X}\right) &= E\left(Z_1^2\right) + iE(Z_1Z_2) - i E(Z_2Z_1) + E\left(Z_2^2\right)=c_{X,X},\\
E(XY) &= E(Z_1Z_3) + i E(Z_2Z_3) + iE(Z_1Z_4) - E(Z_2Z_4) = 0,\\
E\left(X\overline{Y}\right) &= E(Z_1Z_3) + i E(Z_2Z_3) - iE(Z_1Z_4) + E(Z_2Z_4) = c_{X,Y},
\end{align*}
and similar computations yield the other equalities.
\end{exam}

We now introduce Construction \ref{cons:twisting-action} into this setting.
\begin{defn}
For $N \in \bN$ an $N \times N$ \emph{random pair of matrices on $(L_\infty(\Omega, \mu), E)$} is a pair $(X^\ell, X^r)$ where $X^\ell$ is a left matrix and $X^r$ is a right matrix with entries from $L_\infty(\Omega, \mu) \in \L(\L_2(\Omega, \mu))$.
\end{defn}

\begin{rem}
Note that a random pair of matrices is not simply a pair of random matrices, but a pair of random matrices with a certain action on elements of $\M_N(L_\infty(\Omega, \mu))$.
\end{rem}

\begin{defn}
Let 
\[
C = \left[  \begin{array}{cc} c_{\ell,\ell} & c_{\ell,r}\\ c_{r,\ell} & c_{r,r}  \end{array} \right]
\]
be a non-singular, positive matrix with real entries.  A \emph{self-adjoint $C$-Gaussian random pair of matrices} is an $N \times N$ random pair of matrices $(X^\ell, X^r)$ on $(L_\infty(\Omega, \mu), E)$ with $X^\ell = L([X^\ell_{i,j}])$ and $X^r = R([X^r_{i,j}])$ where
\begin{enumerate}
\item $X^k_{i,j} \in L_\infty(\Omega, \mu)$ for all $k \in \{\ell, r\}$ and $i,j \in \{1,\ldots, N\}$,
\item $\overline{X^k_{i,j}} = X^k_{j,i}$ for all $k \in \{\ell, r\}$ and $i,j \in \{1,\ldots, N\}$, and
\item $\left\{X^k_{i,i}, \Re(X^k_{i,j}), \Im(X^k_{i,j}) \, \mid \, k \in \{\ell, r\}, i,j \in \{1, \ldots, N\}, i < j\right\}$ is a Gaussian family such that
\[
E(X^{k_1}_{i,j} X^{k_2}_{l,m}) = \frac{1}{N} \delta_{i,m} \delta_{j,l} c_{k_1, k_2}.
\]
\end{enumerate}  
\end{defn}

\begin{rem}
\label{rem:existence-of-covariance-matrix}
Given a non-singular, positive matrix with real entries
\[
C = \left[  \begin{array}{cc} c_{\ell,\ell} & c_{\ell,r}\\ c_{r,\ell} & c_{r,r}  \end{array} \right]
\]
one can always construct an $N \times N$ self-adjoint $C$-Gaussian random pair of matrices by Definition \ref{defn:Gaussian} and by taking direct sums of the covariance matrices from Examples \ref{exam:diag} and \ref{diag:off-diag}.
\end{rem}

\begin{rem}
Due to commutativity and positivity requirements, only certain bi-free central limit distributions $(\{Z_i\}_{i \in I}, \{Z_j\}_{j \in J})$ on $(\A, \varphi)$ will be limits distributions of random pairs of matrices.  Indeed we will require that $(\A,\varphi)$ is a $*$-non-commutative probability space, $\varphi$ is a positive linear functional, and the covariance matrix of $(\{Z_i\}_{i \in I}, \{Z_j\}_{j \in J})$ is positive, non-singular, and has real entries.  In particular, each $Z_i$ and $Z_j$ must be a non-zero self-adjoint semicircular variable.  Furthermore, as the covariance matrix is self-adjoint with real entries, $\kappa_\chi(Z_{k_1}, Z_{k_2}) = \kappa_{\chi'}(Z_{k_2}, Z_{k_1})$ for all $k_1, k_2 \in I \sqcup J$.  Since a joint moment of $(\{Z_i\}_{i \in I}, \{Z_j\}_{j \in J})$ can be computed as sums of products of second-order $(\ell, r)$-cumulants by the definitions and results of Section \ref{sec:background}, one can verify that $(\{Z_i\}_{i \in I}, \{Z_j\}_{j \in J})$ commute in distribution.  If $C$ is the covariance of $(\{Z_i\}_{i \in I}, \{Z_j\}_{j \in J})$, we call $(\{Z_i\}_{i \in I}, \{Z_j\}_{j \in J})$ a \emph{$C$-bi-free central limit distribution}.
\end{rem}

\begin{thm}
\label{thm:bi-free-central-limit}
Let 
\[
C = \left[  \begin{array}{cc} c_{\ell,\ell} & c_{\ell,r}\\ c_{r,\ell} & c_{r,r}  \end{array} \right]
\]
be a non-singular, positive matrix with real entries and let $(S_\ell, S_r)$ be a $C$-bi-free central limit distribution with respect to $\psi$. For each $N \in \bN$ let $(X^\ell(N), X^r(N))$ be an $N \times N$ self-adjoint $C$-Gaussian random pair of matrices.  Then the joint distribution of $(X^\ell(N), X^r(N))$ with respect $\frac{1}{N} \tr \circ E_{\L(\X_N)}$ tends to the joint distribution of $(S_\ell, S_r)$ with respect to $\psi$ as $N$ tends to infinity; that is, for every $n \in \bN$ and every $\chi : \{1,\ldots, n\} \to \{\ell, r\}$,
\[
\lim_{N \to \infty} \frac{1}{N} \tr\left(E_{\L(\X_N)}\left(X^{\chi(1)}(N) \cdots X^{\chi(n)}(N)\right)\right) = \psi\left(S_{\chi(1)} \cdots S_{\chi(n)}\right).
\]
\end{thm}
\begin{proof}
The proof presented here is motivated by the proofs in \cite{NS2006}*{Lecture 22}.  Notice, by Lemma \ref{lem:expanding-matrices}, Remark \ref{rem:indices-that-matter}, and Theorem \ref{thm:Wick}, that 
\begin{align}
\frac{1}{N} \tr \left(E_{\L(\X)}\left(X^{\chi(1)} \cdots X^{\chi(n)}\right)\right)  
&= \frac{1}{N^{\frac{n}{2}+1}} \sum^N_{i_1, \ldots, i_n=1} E\left(X^{\chi(1)}_{i_{1}, i_{s_\chi(s^{-1}_\chi(1) + 1)}}X^{\chi(2)}_{i_{2}, i_{s_\chi(s^{-1}_\chi(2) + 1)}} \cdots X^{\chi(n)}_{i_{n}, i_{s_\chi(s^{-1}_\chi(n) + 1)}}  \right) \nonumber \\
&= \frac{1}{N^{\frac{n}{2}+1}} \sum^N_{i_1, \ldots, i_n=1} \sum_{\pi\in \P_2(n)} \prod_{\{x,y\} \in \pi} E\left(X^{\chi(x)}_{i_{x}, i_{s_\chi(s^{-1}_\chi(x) + 1)}}X^{\chi(y)}_{i_{y}, i_{s_\chi(s^{-1}_\chi(y) + 1)}} \right) \nonumber \\
&= \frac{1}{N^{\frac{n}{2}+1}} \sum^N_{j_1, \ldots, j_n=1} \sum_{\pi\in \P_2(n)} \prod_{\{x,y\} \in \pi} E\left(X^{\chi(x)}_{j_{s_\chi^{-1}(x)}, j_{s^{-1}_\chi(y) + 1}}X^{\chi(y)}_{j_{s_\chi^{-1}(y)}, j_{s^{-1}_\chi(x) + 1}} \right) \label{eq:wick-mess}
\end{align}
(where the last line follows by replacing $i_k$ with $j_{s^{-1}_\chi(k)}$).
Notice equation (\ref{eq:wick-mess}) is zero unless $n=2m$, in which case it equals
\begin{align*}
& \frac{1}{N^{1+m}} \sum_{\pi\in \P_2(2m)} \sum^N_{j_1, \ldots, j_{2m}=1}  \prod_{\{x,y\} \in \pi} \delta_{j_{s_\chi^{-1}(x)}, j_{s^{-1}_\chi(y) + 1}}\delta_{j_{s_\chi^{-1}(y)}, j_{s^{-1}_\chi(x) + 1}} c_{\chi(x), \chi(y)} \\
&= \frac{1}{N^{1+m}} \sum_{\pi\in \P_2(2m)} \sum^N_{j_1, \ldots, j_{2m}=1}  \prod_{\{s_\chi(x),s_\chi(y)\} \in \pi} \delta_{j_{x}, j_{y + 1}}\delta_{j_{y}, j_{x + 1}} c_{\chi(s_\chi(x)), \chi(s_\chi(y))} \\
&= \frac{1}{N^{1+m}} \sum_{\pi\in \P_2(2m)} \sum^N_{j_1, \ldots, j_{2m}=1}  \prod_{\{x,y\} \in s_\chi^{-1} \cdot \pi} \delta_{j_{x}, j_{y + 1}}\delta_{j_{y}, j_{x + 1}} c_{\chi(s_\chi(x)), \chi(s_\chi(y))}.
\end{align*}

The computations on \cite{NS2006}*{page 365} demonstrate for a fixed $\pi \in \P_2(2m)$ that
\[
\lim_{N \to \infty} \frac{1}{N^{1+m}} \sum^N_{j_1, \ldots, j_{2m}=1}  \prod_{\{x,y\} \in \pi} \delta_{j_{x}, j_{y + 1}}\delta_{j_{y}, j_{x + 1}}  = 0
\]
unless $\pi$ is non-crossing in which case the limit is 1.  Consequently, for a fixed $\pi \in \P_2(2m)$, 
\begin{align}
\lim_{N \to \infty} \frac{1}{N^{1+m}} \sum^N_{i_1, \ldots, i_{2m}=1}  \prod_{\{x,y\} \in s_\chi^{-1} \cdot \pi} \delta_{i_{x}, i_{y + 1}}\delta_{i_{y}, i_{x + 1}} c_{\chi(s_\chi(x)), \chi(s_\chi(y))} = 0 \label{eq:limit-for-commutative}
\end{align}
unless $s^{-1}_\chi \cdot \pi$ is non-crossing pair partition, which means $\pi \in BNC(\chi)$ is a pair partition.  For a pair partition $\pi \in BNC(\chi)$, the limit in equation (\ref{eq:limit-for-commutative}) is $\prod_{\{x,y\} \in \pi} c_{\chi(x), \chi(y)}$.   
Therefore, if $BNC_2(\chi)$ denotes the bi-non-crossing pair partitions corresponding to $\chi$, then
\begin{align*}
\lim_{N \to \infty}  \frac{1}{N} \tr\left(E_{\L(\X_N)}\left(X^{\chi(1)}(N) \cdots X^{\chi(n)}(N)\right)\right) 
&= \sum_{\pi \in BNC_2(\chi)} \prod_{\{x,y\} \in \pi} c_{\chi(x), \chi(y)} \\
&= \sum_{\pi \in BNC_2(\chi)} \kappa_\pi(S_{\chi(1)}, S_{\chi(2)}, \ldots, S_{\chi(n)}) \\
&= \varphi\left(S_{\chi(1)} \cdots S_{\chi(n)}\right). \qedhere
\end{align*}
\end{proof}

Using the above ideas, asymptotic bi-freeness of random pairs of matrices is easily obtained via observing the correct `colouring'.

\begin{thm}
\label{thm:ass-bi-free}
Fix an index set $K$.  For each $k \in K$ let $C_k$ be a $2 \times 2$ non-singular, positive matrix with real entries and let $\{(S_{\ell, k}, S_{r, k})\}_{k \in K}$ be a collection of bi-free two-faced pairs with respect to $\psi$ where $(S_{\ell, k}, S_{r, k})$ is a $C_k$-bi-free central limit distribution. For each $N \in \bN$ and $k \in K$ let $(X^{\ell,k}(N), X^{r,k}(N))$ be an $N \times N$ self-adjoint $C_k$-Gaussian random pair of matrices such that entries are independent for different $k \in K$.  

For every $n \in \bN$, every $\chi : \{1,\ldots, n\} \to \{\ell, r\}$, and every $\epsilon : \{1,\ldots, n\} \to K$,
\[
\lim_{N \to \infty} \frac{1}{N} \tr\left(E_{\L(\X_N)}\left(X^{\chi(1), \epsilon(1)}(N) \cdots X^{\chi(n), \epsilon(n)}(N)\right)\right) = \psi\left(S_{\chi(1), \epsilon(1)} \cdots S_{\chi(n), \epsilon(n)}\right).
\]
In particular, $\{(X^{\ell,k}(N), X^{r,k}(N))\}_{k \in K}$ are asymptotically bi-free with respect $\frac{1}{N} \tr \circ E_{\L(\X_N)}$.
\end{thm}
\begin{proof}
By repeating the ideas of Theorem \ref{thm:bi-free-central-limit}, we obtain
\begin{align*}
\frac{1}{N} \tr &\left(E_{\L(\X)}\left(X^{\chi(1), \epsilon(1)} \cdots X^{\chi(n), \epsilon(n)}\right)\right) \\
&= \frac{1}{N^{\frac{n}{2} + 1}} \sum_{\pi\in \P_2(n)} \sum^N_{j_1, \ldots, j_{n}=1}  \prod_{\{x,y\} \in \pi} \delta_{j_{s_\chi^{-1}(x)}, j_{s^{-1}_\chi(y) + 1}}\delta_{j_{s_\chi^{-1}(y)}, j_{s^{-1}_\chi(x) + 1}} \delta_{\epsilon(x), \epsilon(y)} c^{\epsilon(x)}_{\chi(x), \chi(y)}.
\end{align*}
The result now follows as in Theorem \ref{thm:bi-free-central-limit} as the only $\pi$ that contribute asymptotically to the sum are $\pi\in BNC_2(\chi)$ that may be coloured correctly.
\end{proof}

\begin{rem}
A more general result than Theorem \ref{thm:ass-bi-free} is easy to obtain.  Indeed let $(\{Z_i\}_{i \in I}, \{Z_j\}_{j \in J})$ be a (centred) bi-free central limit distribution and let $C = (c_{k,m})_{k,m \in I \sqcup J}$ be the $(I \sqcup J) \times (I \sqcup J)$ covariance matrix.

If $C$ is non-singular, positive matrix with real entries, then similar arguments to those given in Remark \ref{rem:existence-of-covariance-matrix} show that for each $N \in \bN$ there exists $\{X^k_{i,j}(N) \, \mid \, i,j \in \{1,\ldots, N\}, k \in I \sqcup J\} \subseteq (L_\infty(\Omega, \mu), E)$ such that
\begin{itemize}
\item $\overline{X^k_{i,j}(N)} = X^k_{j,i}(N)$ for all $k \in I \sqcup J$ and $i,j \in \{1,\ldots, N\}$, and
\item $\left\{X^k_{i,i}(N), \Re(X^k_{i,j}(N)), \Im(X^k_{i,j}(N)) \, \mid \, i,j \in \{1, \ldots, N\}, i < j, k \in I \sqcup J \right\}$ is a Gaussian family such that
\[
E\left(X^{k_1}_{i,j}(N) X^{k_2}_{l,m}(N)\right) = \frac{1}{N} \delta_{i,m} \delta_{j,l} c_{k_1, k_2}.
\]
\end{itemize}

For each $N \in \bN$ and $k \in I \sqcup J$, let $X^k(N) = Z([X^k_{i,j}])$ where $Z = L$ if $k \in I$ and $Z = R$ if $k \in J$.
Similar arguments to those above show that the asymptotic joint distribution of $(\{X^i(N)\}_{i \in I}, \{X^j(N)\}_{j \in J})$ with respect to $\frac{1}{N} \tr \circ E_{\L(\X_N)}$ as $N \to \infty$ is equal to the joint distribution of $(\{Z_i\}_{i \in I}, \{Z_j\}_{j \in J})$.  Theorem \ref{thm:ass-bi-free} then follows  by using the correct covariance matrix.

\end{rem}

For those familiar with the development of bi-free probability, it should not be a surprise that results from free probability can be generalized to bi-free probability provided all left operators commute with all right operators.  Indeed the following result can be interpreted as ``in the commutative world, asymptotic bi-freeness is pretty much asymptotic freeness."

\begin{thm}
\label{thm:ass-free-to-ass-bi-free}
Let $K$ be a fixed set and consider $(\L(\X), \varphi_\X)$ for some pointed vector space $\X$. For each $N \in \bN$ let $\{X_k(N)\}_{k \in K}$ be left $N \times N$ matrices and let $\{Y_k(N)\}_{k \in K}$ be right $N \times N$ matrices of $\L(\X)$.  Suppose that the joint distributions of $\{X_k(N)\}_{k \in K} \cup \{Y_k(N)\}_{k \in K}$ converge with respect to $\frac{1}{N} \tr \circ E_{\L(\X_N)}$.  Furthermore, suppose
\begin{enumerate}
\item $X_k(N) Y_m(N) = Y_m(N) X_k(N)$ for all $k,m \in K$ and for all $N$, and
\item $X_k(N) I_N = Y_k(N) I_N$ for all $k \in K$ and $N$.
\end{enumerate}
If $K = K_1 \sqcup K_2$ then $\{X_k(N)\}_{k \in K_1}$ is asymptotically free from $\{X_k(N)\}_{k \in K_2}$ with respect to $\frac{1}{N} \tr \circ E_{\L(\X_N)}$ if and only if $(\{X_k(N)\}_{k \in K_1}, \{Y_k(N)\}_{k \in K_1})$ is asymptotically bi-free from $(\{X_k(N)\}_{k \in K_2}, \{Y_k(N)\}_{k \in K_2})$ with respect to $\frac{1}{N} \tr \circ E_{\L(\X_N)}$.
\end{thm}
\begin{proof}
The proof of this result is obtained by demonstrating that the universal bi-free moment polynomials (i.e. equation (\ref{eq:universal-polys})) hold asymptotically for $(\{X_k(N)\}_{k \in K_1}, \{Y_k(N)\}_{k \in K_1})$ with $(\{X_k(N)\}_{k \in K_2}, \{Y_k(N)\}_{k \in K_2})$ if and only if the universal free polynomials (those in equation (\ref{eq:universal-polys}) restricted to $\chi$ with $\chi(m) = \ell$ for all $m$) hold asymptotically for $\{X_k(N)\}_{k \in K_1}$ with $\{X_k(N)\}_{k \in K_2}$.  A carbon copy of the proof of \cite{CNS2014-2}*{Theorem 10.2.1} with $\lim_{N \to \infty}$ inserted in multiple places yields the result.  We omit further details as similar arguments will be used later in the paper (see Example \ref{exam:poisson} and the proof of Theorem \ref{thm:asymptotic-bi-free-fock}).
\end{proof}
\begin{rem}
Note the condition ``$\{X_k(N)\}_{k \in K_1}$ is asymptotically free from $\{X_k(N)\}_{k \in K_2}$ with respect to $\frac{1}{N} \tr \circ E_{\L(\X_N)}$"  in Theorem \ref{thm:ass-free-to-ass-bi-free} is precisely saying that $\{X_k(N)\}_{k \in K_1}$ is asymptotically free from $\{X_k(N)\}_{k \in K_2}$ as usual matrices of operators.  As such, since random matrices have commutative entries, many results from random matrices immediate have random pairs of matrices analogues.  Furthermore, if a result can be proved for random matrices grouped into pairs of two, one immediately has a random pair of matrices result by viewing the first element in each pair as a left matrix and the second element in the pair as a right matrix.
\end{rem}

\begin{exam}
\label{exam:poisson}
As an example application of how Theorem \ref{thm:ass-free-to-ass-bi-free} and its proof work, we will demonstrate for every bi-free Poisson distribution $\mu$ from \cite{GHM2015}*{Example 3.11b} a random pair of matrices with limit distribution $\mu$.  Fix a number $\lambda \in (0, 1)$ and for each $N \in \bN$ choose $M_N \in \bN$ such that $\lim_{N \to \infty} \frac{M_N}{N} = \lambda$.  For each $N \in \bN$ let $Y(N)$ be an $N \times M_N$ random matrix whose entries are independent Gaussian random variables with mean 0 and variance 1.  If 
\[
X(N) = \frac{1}{M_N} Y(N)^* Y(N)
\]
(a Wishart matrix) then the Marchenko-Pastur law implies the limit distribution of $X(N)$ with respect to $\frac{1}{N} \tr \circ E_{\L(\X_N)}$ is equal to the free Poisson law $\mu_P$ with rate $\lambda$ and jump size 1.  Recall the $n^{\mathrm{th}}$ free cumulants of $\mu_P$ is $\lambda$ for all $n \in \bN$ so the $n^{\mathrm{th}}$ free cumulant of $X(N)$ with $\frac{1}{N} \tr \circ E_{\L(\X_N)}$ tends to $\lambda$ as $N$ tends to infinity.

For each $(\alpha, \beta) \in \bR^2$, \cite{GHM2015}*{Example 3.11b} defined the\emph{ bi-free Poisson distribution with rate $\lambda$ and jump size} $(\alpha, \beta)$, denoted $\mu_{bP, \alpha, \beta}$, to be the limit distribution of
\[
\left(\left(1-\frac{\lambda}{N}\right) \delta_{(0,0)} + \frac{\lambda}{N} \delta_{(\alpha, \beta)} \right)^{\boxplus\boxplus N}
\]
where $\boxplus\boxplus$ denotes the bi-free additive convolution.  Furthermore \cite{GHM2015} demonstrates that if $\chi : \{1,\ldots, n\} \to \{\ell, r\}$ then the $(\ell, r)$-cumulant of $\mu_{bP, \alpha, \beta}$ with respect to $\chi$ is $\lambda \alpha^{|\chi^{-1}(\{\ell\})|} \beta^{|\chi^{-1}(\{r\})|}$.

If, for each $N \in \bN$, $X^\ell(N) = \alpha L(X(N))$ and $X^r(N) = \beta R(X(N))$, we claim the joint distribution of the random pair of matrices $(X^\ell(N), X^r(N))$ with respect to $\Psi_N := \frac{1}{N} \tr \circ E_{\L(\X_N)}$ tends to $\mu_{bP, \alpha, \beta}$ as $N \to \infty$.  Indeed, by the moment-cumulant equations (\ref{eq:cum-in-terms-of-mom}) and (\ref{eq:mom-in-terms-of-cum}) it suffices to show  for all $\chi : \{1,\ldots, n\} \to \{\ell, r\}$ that
\[
\lim_{N \to \infty} \kappa_\chi\left(X^{\chi(1)}(N), \ldots, X^{\chi(n)}(N)\right) = \lambda \alpha^{|\chi^{-1}(\{\ell\})|} \beta^{|\chi^{-1}(\{r\})|}.
\]

Notice that if $\chi_\ell, \chi_r : \{1,\ldots, n\} \to \{\ell, r\}$ are defined by $\chi_\ell(k) =\ell$ and $\chi_r(k) = r$ for all $k$ then
\begin{align*}
\lim_{N \to \infty}\kappa_{\chi_\ell}\left(X^{\chi(1)}(N), \ldots, X^{\chi(n)}(N)\right) &= \lim_{N \to \infty}\kappa_n(\alpha X(N), \ldots, \alpha X(N))  =\lambda \alpha^n \text{ and} \\
\lim_{N \to \infty}\kappa_{\chi_r}\left(X^{\chi(1)}(N), \ldots, X^{\chi(n)}(N)\right) &= \lim_{N \to \infty}\kappa_n(\beta X(N), \ldots, \beta X(N))  =\lambda \beta^n
\end{align*}
since $\Psi_N(L(X(N))^n) = \Psi_N(R(X(N))^n) = \Psi_N(X(N)^n)$ (the later computed traditionally).  Therefore, if $\alpha = 0$ or $\beta = 0$, the proof is complete.  Thus we assume that $\alpha, \beta \neq 0$ and $\chi \neq \chi_\ell, \chi_r$.

Let $s$ be the permutation such that 
\[
\chi^{-1}(\{\ell\}) = \{s(1) < \ldots < s(k)\} \qqand \chi^{-1}(\{ r\}) = \{s(k+1) < \ldots < s(n)\}.
\]
Let $\hat\chi = \chi\circ s$. 
Note that replacing $\chi$ by $\hat\chi$ corresponds to moving all the right nodes down to be beneath the left ones, without changing their relative order.  Furthermore $s$ induces a natural isomorphism from $BNC(\chi)$ to $BNC(\hat\chi)$ via $\pi \mapsto s^{-1} \cdot \pi$ such that $\mu_{BNC}(\pi, 1_\chi) = \mu_{BNC}(s^{-1} \cdot \pi, 1_{\hat\chi})$ as $\mu_{BNC}$ is completely determined by the lattice structure (see Definition \ref{defn:mobius}).  Therefore, since $X^\ell(N)$ and $X^r(N)$ commute, we obtain that
\begin{align*}
\kappa_\chi\left(X^{\chi(1)}(N), \ldots, X^{\chi(n)}(N)\right) &= \sum_{\pi \in BNC(\chi)} \Psi_{N, \pi}\left(X^{\chi(1)}(N), \ldots, X^{\chi(n)}(N)\right) \mu_{BNC}(\pi, 1_\chi) \\
&= \sum_{\pi \in BNC(\hat\chi)} \Psi_{N, \pi}(\underbrace{X^{\ell}(N), \ldots, X^{\ell}(N)}_{|\chi^{-1}(\{\ell\})| \text{ copies}}, \underbrace{X^{r}(N), \ldots, X^{r}(N)}_{|\chi^{-1}(\{r\})| \text{ copies}} ) \mu_{BNC}(\pi, 1_{\hat\chi}).
\end{align*}
Let $\chi' : \{1,\ldots, n\} \to \{\ell, r\}$ be defined by $\chi'(n) = \ell$ and $\chi'(k) = \hat\chi(k)$.  Note that replacing $\hat\chi$ by $\chi'$ corresponds to moving the bottom right node to the left side and thus induced a natural isomorphism from $BNC(\hat\chi)$ to $BNC(\chi')$.  Since $X^r(N) I_N = \frac{\beta}{\alpha} X^\ell(N) I_N$, we obtain, as in \cite{CNS2014-2}*{Theorem 10.2.1}, that
\begin{align*}
\kappa_\chi & \left(X^{\chi(1)}(N), \ldots, X^{\chi(n)}(N)\right) \\ 
&= \frac{\beta}{\alpha} \sum_{\pi \in BNC(\chi')} \Psi_{N, \pi}(\underbrace{X^{\ell}(N), \ldots, X^{\ell}(N)}_{|\chi^{-1}(\{\ell\})| \text{ copies}}, \underbrace{X^{r}(N), \ldots, X^{r}(N)}_{|\chi^{-1}(\{r\})|-1 \text{ copies}}, X^\ell(N)  ) \mu_{BNC}(\pi, 1_{\chi'}).
\end{align*}
By repeating the above arguments (i.e. commuting $X^\ell(N)$ past all the $X^r(N)$ in the above expression and then changing the last $X^r(N)$ to a $X^\ell(N)$), we obtain that
\begin{align*}
\lim_{N \to \infty} \kappa_\chi\left(X^{\chi(1)}(N), \ldots, X^{\chi(n)}(N)\right) &= \lim_{N \to \infty} \left(\frac{\beta}{\alpha}  \right)^{|\chi^{-1}(\{r\})|} \kappa_{\chi_\ell}\left(X^{\chi(1)}(N), \ldots, X^{\chi(n)}(N)\right) \\
&=  \left(\frac{\beta}{\alpha}  \right)^{|\chi^{-1}(\{r\})|} \lambda \alpha^n = \lambda \alpha^{|\chi^{-1}(\{\ell\})|} \beta^{|\chi^{-1}(\{r\})|}
\end{align*}
as desired.
\end{exam}

We will generalize a few more results from random matrix theory to the bi-free setting by using Theorem \ref{thm:ass-free-to-ass-bi-free}.  Of course these results are not exhaustive and there are many more applications and generalizations of random matrix results.

For the first result, we will use the ideas of \cite{NS2006}*{Lecture 22} to demonstrate the asymptotic bi-freeness from Gaussian random  pairs of matrices and ``constant matrices''.    By constant matrices, we mean elements of $\M_N(\bC) \subseteq \M_N(\L_\infty(\Omega, \mu))$, say $\{D_m(N)\}_{m \in K'}$, such that there exists a non-commutative probability space $(\A, \varphi)$ and elements $\{d_m\}_{m \in K'} \subseteq \A$ such that the joint distributions of $\{D_m(N)\}_{m \in K'}$ with respect to $\frac{1}{N} \tr$ tend to the joint distribution of  $\{d_m\}_{m \in K'}$.  Often one takes $D_m(N)$ to be diagonal matrices as such matrices can approximate the distribution of any measure on $\bC$.

The following result is immediately obtain from the proof of \cite{NS2006}*{Theorem 22.35} by running said argument with pairs of Gaussian random matrices independent from one another instead of just independent Gaussian random matrices - something that had little value to write down in the past.  As the proof is modified simply by keeping track of cumulant terms from non-crossing pair partitions (e.g. see the proof of Theorem \ref{thm:bi-free-central-limit}), we omit the details.  For those unhappy with this omission, the results of Sections \ref{sec:fock} and \ref{sec:q-deformed} can be used to prove asymptotic bi-freeness from `constant diagonal matrices' (see Remark \ref{rem:q-matrix-results}).

\begin{thm}
For each $N \in \bN$ let $\{(Z_k(N), Z'_k(N))\}_{k \in K}$ be pairs of $N \times N$ Gaussian random matrices such that $\{Z_k(N), Z'_k(N)\}$ is independent from $\{Z_m(N), Z'_m(N)\}$ if $k \neq m$ and the covariance matrix of $(Z_k(N), Z'_k(N))$ is a positive, non-singular matrix $C_k$ with real entries.  For each $N \in \bN$ let $\{D_m(N)\}_{m \in K'} \subseteq \M_N(\bC)$ be as above.

The joint distribution of 
\[
\{(Z_k(N), Z'_k(N))\}_{k \in K} \cup \{D_m(N)\}_{m \in K'}
\]
with respect to $\frac{1}{N} \tr \circ E_{\L(\X_N)}$ tends to the joint distribution of 
\[
\{(s_k, s'_k)\}_{k \in K} \cup \{d_m\}_{m \in K'}
\]
where $\{(s_k, s'_k)\}_{k \in K}$ is free from $\{d_m\}_{m \in K'}$, $(s_k, s'_k)$ is independent from $(s_m, s'_m)$ if $k \neq m$, and $(s_k, s'_k)$ is a pair of self-adjoint semicircular variables with covariance matrix $C_k$.  Consequently, Theorem \ref{thm:ass-free-to-ass-bi-free} implies $\{(L(Z_k(N)), R(Z'_k(N)))\}_{k \in K}$ are asymptotically bi-free central limit distributions and are asymptotically bi-free from $(\{L(D_m(N))\}_{m \in K'}, \{R(D_m(N))\}_{m \in K'})$.
\end{thm}

For our final application of Theorem \ref{thm:ass-free-to-ass-bi-free}, we will examine Haar bi-unitary random pairs of matrices.  Recall from \cite{CNS2014-2}*{Section 10} that a pair $(U_\ell, U_r)$ in a non-commutative probability space $(\A, \varphi)$ is said to be a \emph{Haar bi-unitary} if the joint distribution of $(U_\ell, U_r)$ with respect to $\varphi$ is equal to the joint distribution of $(U, U)$ where $U$ is a Haar unitary in some non-commutative probability space.  Consequently, if $U(N) \in \M_N(L_\infty(\Omega, \mu))$ is a Haar unitary random matrix, we will call $(L(U(N)), R(U(N)))$ a \emph{Haar bi-unitary random pair of matrices}.

Based on definition and \cite{NS2006}*{Theorem 23.13}, we immediately obtain the following.
\begin{thm}
\label{thm:Haar}
For each $N \in \bN$ let $\{U_k(N)\}_{k \in K}$ be independent $N \times N$ Haar unitary random matrices and let $\{D_m(N)\}_{m \in K'} \subseteq \M_N(\bC)$ be such that for some $*$-non-commutative probability space $(\A, \varphi)$ and for some elements $\{d_m\}_{m \in K'} \subseteq \A$ the joint $*$-distributions of $\{D_m(N)\}_{m \in K'}$ with respect to $\frac{1}{N} \tr$ tend to the joint $*$-distribution of  $\{d_m\}_{m \in K'}$.  Then the joint $*$-distributions of 
\[
\{(L(U_k(N)), R(U_k(N)))\}_{k \in K} \cup (\{L(D_m(N))\}_{m \in K'}, \{R(D_m(N))\}_{m \in K'})
\]
with respect to $\frac{1}{N} \tr \circ E_{\L(\X_N)}$ tends to the joint $*$-distribution of the bi-free pairs of faces 
\[
\{(U_{\ell,k}, U_{r,k})\}_{k\in K} \cup     (\{L(d_m)\}_{m \in K'}, \{R(d_m)\}_{m \in K'})
\]
as $N \to \infty$, where each $(U_{\ell,k}, U_{r,k})$ is a Haar bi-unitary.  In particular, Haar bi-unitary random pair of matrices are asymptotically bi-free from constant matrices.
\end{thm}

\begin{rem}
Recall \cite{CNS2014-2}*{Theorem 10.1.3} shows that if $(C, D)$ is a pair of algebras that is bi-free from a Haar bi-unitary $(U_\ell, U_r)$ in a non-commutative probability space $(\A, \varphi)$, then $(C, D)$ and $(U_\ell C U_\ell^*, U_r D U_r^*)$ are bi-free pairs of faces.  In particular, combining \cite{CNS2014-2}*{Theorem 10.1.3} with Theorem \ref{thm:Haar}, by conjugating constant matrices by Haar bi-unitary random pairs of matrices, one can obtain many bi-free joint distributions.
\end{rem}

\section{Bi-Matrix Models with Fock Space Entries}
\label{sec:fock}

In this section, a bi-matrix model involving left and right creation and annihilation operators on a Fock space will be examined.  The results of this section generalize those of \cite{S1997}*{Section 5} to the bi-free setting and provide bi-matrix models for all bi-free central limit distributions.

Throughout this section let $\X = \F(\H)$, namely the Fock space of a Hilbert space $\H$ of sufficiently large size.   If $\Omega$ is the vacuum vector of $\F(\H)$ then $p_\X : \X \to \bC$ is defined by $p_\X( \lambda \Omega \oplus \eta) = \lambda$.  Furthermore $\varphi_0 := \varphi_\X : \L(\X) \to \bC$ is defined by $\varphi_0(T) = \langle T \Omega, \Omega \rangle$ (all inner products linear in the first entry) thereby making $(\L(\X), \varphi_0)$ a non-commutative probability space.

Given an element $h \in \H$, let $l(h)$ and $r(h)$ denote the left and right creation operators by $h$ respectively and let $l^*(h) = (l(h))^*$ and $r^*(h) = (r(h))^*$ be the corresponding annihilation operators.  Then:
\begin{align}
l^*(h_1)l(h_2) &= \langle h_2, h_1\rangle I_{\X},\label{eq:fock-left-identity}\\
r^*(h_1)r(h_2) &= \langle h_2, h_1\rangle I_{\X}, \label{eq:fock-right-identity} \\
[l^*(h_1), r(h_2)] &= l^*(h_1) r(h_2) - r(h_2) l^*(h_1) = \langle h_2, h_1 \rangle P_\Omega,\label{eq:fock-left-right-*} \\
[r^*(h_1), l(h_2)] &= r^*(h_1) l(h_2) - l(h_2) r^*(h_1) = \langle h_2, h_1 \rangle P_\Omega, \label{eq:fock-right-left-*} \text{ and} \\
[l(h_1), r(h_2)] &= [l^*(h_1), r^*(h_2)] = 0, \label{eq:fock-commuting}
\end{align}
where $P_\Omega \in \L(\X)$ is the projection onto the vacuum vector.  Furthermore, for all $T, S \in \L(\X)$, 
\[
\varphi(TP_\Omega S) = \langle T P_\Omega S \Omega, \Omega \rangle = \varphi(S) \langle T \Omega, \Omega\rangle = \varphi(T) \varphi(S).
\]

 Using notation and conventions from Construction \ref{cons:twisting-action}, we have the following which is the bi-free analogue of part of \cite{S1997}*{Theorem 5.2}.
\begin{thm}
\label{thm:matrix-*-distributions}
Given an index set $K$, an $N \in \bN$, and an orthonormal set of vectors $\{h^k_{i,j} \, \mid \, i,j \in \{1, \ldots, N\}, k \in K\} \subseteq \H$ consider the $N\times N$ left and right matrices of $\L(\X_N)$ defined for each $k \in K$ by
\begin{align*}
L_k :=  \frac{1}{\sqrt{N}}L\left( \left[ l(h^k_{i,j})\right] \right), \quad L^*_k := \frac{1}{\sqrt{N}} L\left( \left[ l^* (h^k_{j,i})\right] \right) \quad
R_k := \frac{1}{\sqrt{N}}  R\left( \left[r(h^k_{i,j})\right] \right), \quad R^*_k := \frac{1}{\sqrt{N}} R\left( \left[ r^*(h^k_{j,i})\right] \right).
\end{align*}
Let $\Phi = \frac{1}{N} \tr \circ E_{\L(\X_N)}$.  Then
\begin{enumerate}
\item $L^*_m L_k = \delta_{k,m} I_{\X_N}$, $R^*_m R_k= \delta_{k,m} I_{\X_N}$, $[L_m, R_k] = [L^*_m, R^*_k] = 0$, $[L^*_m, R_k] = \delta_{k,m} P_0$, and $[R^*_m, L_k] = \delta_{k,m} P_0$ where $P_0 \in \L(\X_N)$ is the linear map
\[
P_0([\xi_{i,j}]) = \frac{1}{N} \diag( P_\Omega( \tr([\xi_{i,j}])), \ldots, P_\Omega( \tr([\xi_{i,j}]))).
\] \label{part:mm-1}
\item $P_0^2 = P_0$.\label{part:mm-2}
\item $\Phi(T P_0 S) = \Phi(T)\Phi(S)$ for all $T, S \in \L(\X_N)$. \label{part:mm-3}
\item The joint distribution of $\{L_k, L^*_k, R_k, R^*_k\}_{k \in K}$ with respect to $\Phi$ is equal the joint distribution of $\{l(h^k), l^*(h^k), r(h^k), r^*(h^k)\}_{k \in K}$ with respect to $\varphi_0$  where $\{h^k\}_{k \in K} \subseteq \H$ is an orthonormal set.\label{part:mm-5}
\end{enumerate}
\end{thm}
\begin{proof}
For (\ref{part:mm-1}), notice
\begin{align*}
L_m^* L_k  = \frac{1}{N} L\left(  \left[  \sum^N_{x=1}   l^*(h^m_{x,i})  l(h^k_{x, j})                \right]          \right) 
  =\frac{1}{N} \delta_{k,m} L\left(  \left[  \sum^N_{x=1}  \delta_{i,j} \right]          \right) 
 = \delta_{k,m}  L(I_{N, \L(\X)}) = \delta_{k,m} I_{\X_N},
\end{align*}
as $L$ is a homomorphism and
\begin{align*}
R_m^* R_k =  \frac{1}{N}  R\left(   \left[ \sum^N_{y=1}   r^*(h^m_{j,y}) r(h^k_{i,y})   \right]    \right)=  \frac{1}{N} \delta_{k,m}  R\left(  \left[  \sum^N_{y=1}   \delta_{i,j} \right]          \right)=  \delta_{k,m}  R(I_{N, \L(\X)}) = \delta_{k,m}  I_{\X_N}
\end{align*}
as $R$ is a homomorphism on $\M_N(\L(\X)^{\op})^{\op}$.  Furthermore equation (\ref{eq:fock-commuting}) together with Remark \ref{rem:commutative} implies that  $[L_m, R_k] = [L^*_m, R^*_k] = 0$.

Since
\begin{align*}
L_m R^*_k [\xi_{i,j}] &= \frac{1}{\sqrt{N}}L_m \left(\left[ \sum^N_{y=1}  r^*(h^k_{j,y}) \xi_{i,y} \right]     \right) = \frac{1}{N} \left[\sum^N_{x,y=1}   l(h^m_{i,x}) r^*(h^k_{j,y}) \xi_{x,y}      \right],
\end{align*}
and since, by equation (\ref{eq:fock-right-left-*}),
\begin{align*}
R^*_k L_m [\xi_{i,j}] = \frac{1}{\sqrt{N}} R^*_k \left( \left[    \sum^N_{x=1}   l(h^m_{i,x}) \xi_{x,j} \right] \right) &= \frac{1}{N}\left[    \sum^N_{x,y=1}         r^*(h^k_{j,y}) l(h^m_{i,x}) \xi_{x,y}    \right] \\
&= \frac{1}{N} \left[    \sum^N_{x,y=1}       \left(    l(h^m_{i,x}) r^*(h^k_{j,y}) + \langle  h^m_{i,x}, h^k_{j,y}\rangle P_\Omega     \right) \xi_{x,y}    \right] \\
&= L_m R^*_k [\xi_{i,j}] + \frac{1}{N}\left[    \sum^N_{x,y=1}       \langle  h^m_{i,x}, h^k_{j,y}\rangle P_\Omega\xi_{x,y}    \right] \\
&= L_m R^*_k [\xi_{i,j}] + \frac{1}{N}\delta_{k,m}  \left[    \sum^N_{x=1}   \delta_{i,j} P_\Omega \xi_{x,x}    \right] \\
&= L_m R^*_k [\xi_{i,j}] + \delta_{k,m}  P_0 [\xi_{i,j}],
\end{align*}
we obtain that $[L^*_m, R_k] = \delta_{k,m} P_0$.  Similar computations show $[R^*_m, L_k] = \delta_{k,m} P_0$ thereby completing part (\ref{part:mm-1}).

Part (\ref{part:mm-2}) is a trivial computation using the fact that $P_\Omega^2 = P_\Omega$.  To see part (\ref{part:mm-3}),  write $S (I_{N, \Omega}) = [\xi_{i,j}]$.  Then
\[
\Phi(S) = \frac{1}{N} \sum^N_{y=1} p_\X(\xi_{y,y})
\]
so
\[
P_0 S (I_{N, \Omega}) = \frac{1}{N} \diag\left( P_\Omega \left( \sum^N_{y=1} \xi_{y,y}\right), \ldots, P_\Omega \left( \sum^N_{y=1} \xi_{y,y}\right)  \right) = \Phi(S) I_{N, \Omega}.
\]
Hence
\[
\Phi(T P_0 S) = p_{\X_N}(TP_0 S (I_{N, \Omega})) = \Phi(S) p_{\X_N}(T I_{N, \Omega}) = \Phi(S) \Phi(T).
\]

For part (\ref{part:mm-5}), it must be shown that given any word $W$ in $\{L_k, L^*_k, R_k, R^*_k\}_{k \in K}$ we have $\Phi(W) = \varphi_0(w)$ where $w$ is the corresponding word in $\left\{l(h^k), l^*(h^k), r(h^k), r^*(h^k)\right\}_{k \in K}$.  Note \cite{S1997}*{Theorem 5.2} completes the claim if no $R_k$ nor $R^*_k$ appear in $W$.  Thus we  proceed by an induction argument using parts (\ref{part:mm-1}) and (\ref{part:mm-3}).

Suppose it has been demonstrated that $\Phi(W) = \varphi_0(w)$ for all words $W$ with at most $n\geq 0$ occurrences of $\{R_k, R^*_k\}_{k \in K}$ and all words $W$ with $n+1$ occurrences $\{R_k, R^*_k\}_{k \in K}$ and length at most $m \geq 0$.  Let $W$ be a word with $n+1$ occurrences of $\{R_k, R^*_k\}_{k \in K}$ and length $m+1$.  

If $W = V R^g_k$ for some word $V$ of length one less than the length of $W$ and some $g \in \{\cdot, *\}$, then, since $R^g_k I_{N, \Omega} = L^g_k I_{N, \Omega}$ (see Remark \ref{rem:transitioning-from-left-to-right}), we obtain that
\[
\Phi(W) = \Phi\left(V R^g_k\right) = \Phi\left(V L^g_k\right).
\]
If $w = v r^g_k$ is the corresponding word in $\left\{l(h^k), l^*(h^k), r(h^k), r^*(h^k)\right\}_{k \in K}$, then since $r^g_k \Omega = l^g_k \Omega$, we obtain that
\[
\varphi_0(w) = \varphi_0\left(v r^g_k\right) = \varphi_0\left(v l^g_k\right).
\]
Thus the inductive hypotheses then implies that $\Phi(W) = \varphi_0(w)$.  

Otherwise we may decompose $W = V_1 R_k^{g_1} L_m^{g_2} V_2$ for some $g_1, g_2 \in\{1, \ast\}$ and some words $V_1, V_2$ with $V_2$ a word in $\{L_k, L_k^*\}_{k \in K}$.  We will now demonstrate a process to reduce the length of $L_m^{g_2} V_2$.  Write $w = v_1 r_k^{g_1} l_m^{g_2} v_2$.  If $g_1 = g_2$ or $g_1 \neq g_2$ and $k \neq m$, then, by part (\ref{part:mm-1}) along with equations (\ref{eq:fock-left-right-*}), (\ref{eq:fock-right-left-*}), and (\ref{eq:fock-commuting}), we obtain
\[
\Phi(W) = \Phi\left( V_1  L_m^{g_2} R_k^{g_1}V_2\right) \qand \varphi_0(w) = \varphi_0\left(v_1  l_m^{g_2} r_k^{g_1} v_2\right).
\]
Otherwise, if $g_1 = *$, $g_2 = 1$, and $k = m$, then
\[
\Phi(W) = \Phi\left(V_1 L_m^{g_2} R_k^{g_1}V_2\right) + \Phi(V_1 P_0 V_2) = \Phi\left(V_1 L_m^{g_2} R_k^{g_1}V_2\right) + \Phi(V_1)\Phi( V_2)
\]
by parts (\ref{part:mm-1}) and (\ref{part:mm-3}) whereas
\[
\varphi_0(w) = \varphi_0\left(v_1  l_m^{g_2} r_k^{g_1} v_2\right) + \varphi_0(v_1  P_\Omega v_2) = \varphi_0\left(v_1  l_m^{g_2} r_k^{g_1} v_2\right) + \varphi_0(v_1) \varphi_0(v_2)
\]
by equation (\ref{eq:fock-right-left-*}).  As similar expressions hold for $g_1 = *$, $g_2 = 1$, and $k = m$ (changing $+$ to $-$), by the inductive hypotheses we obtain that $\Phi(W) = \varphi_0(w)$ provided $\Phi\left(V_1  L_m^{g_2} R_k^{g_1}V_2\right) = \varphi_0\left(v_1  l_m^{g_2} r_k^{g_1} v_2\right)$.   By repeating this process, we eventually reduce the word $W$ into one ending with $R_k^{g_1}$ thereby completing the inductive step by previous arguments.
\end{proof}

\begin{rem}
As all bi-free central limit distributions may be obtained using left and right creation and annihilation operators on a Fock space by \cite{V2014}*{Theorem 7.4}, Theorem \ref{thm:matrix-*-distributions} may be used to construct bi-matrix analogues of all bi-free central limit distributions by taking suitable linear combinations.
\end{rem}

With Theorem \ref{thm:matrix-*-distributions} complete, we now demonstrate the bi-free analogue of the second half of \cite{S1997}*{Theorem 5.2}.
\begin{thm}
\label{thm:asymptotic-bi-free-fock}
With the notation in Theorem \ref{thm:matrix-*-distributions}, 
\[
(\{L_k, L^*_k\}_{k \in K} , \{R_k, R^*_k\}_{k \in K}) \qand (L(\M_N(\bC)), R(\M_N(\bC)^{\op}))
\]
are bi-free with respect to $\Phi$.
\end{thm}
\begin{proof}
To prove the claim, it suffices by Theorem \ref{thm:universal-moment-bi-free} to show that the universal bi-free moment polynomials from equation (\ref{eq:universal-polys}) are satisfied; that is, we must demonstrate that
\begin{align}
\Phi\left(Z^{\chi(1), \epsilon(1)}_{k_1} \cdots Z^{\chi(n), \epsilon(n)}_{k_n}\right) = \sum_{ \pi \in BNC(\chi) } \left[ \sum_{\substack{\sigma \in BNC(\chi)  \\ \pi \leq \sigma \leq \epsilon }} \mu_{BNC}(\pi, \sigma)  \right] \Phi_\pi\left(Z^{\chi(1), \epsilon(1)}_{k_1}, \ldots, Z^{\chi(n), \epsilon(n)}_{k_n} \right) \label{eq:up-to-verify}
\end{align}
for all $\chi : \{1, \ldots, n\} \to \{\ell, r\}$ and $\epsilon : \{1,\ldots, n\} \to \{1,2\}$ where
\[
Z^{\ell, 1}_j \in \{L_k, L^*_k, I_{\X_N}\}_{k \in K}, \quad Z^{r, 1}_j  \in \{R_k, R^*_k, I_{\X_N}\}_{k \in K}, \quad Z^{\ell, 2}_j \in L(\M_N(\bC)), \quad Z^{r, 2}_j  \in R(\M_N(\bC)^{\op}).
\]

Suppose equality has been demonstrated in equation (\ref{eq:up-to-verify}) for all $\chi : \{1, \ldots, n'\} \to \{\ell, r\}$ and $\epsilon : \{1,\ldots, n'\} \to \{1,2\}$ with $|\chi^{-1}(\{r\})| \leq m$ and for all $\chi : \{1, \ldots, n'\} \to \{\ell, r\}$ and $\epsilon : \{1,\ldots, n'\} \to \{1,2\}$ with $|\chi^{-1}(\{r\})| = m+1$ and $n' < n$.  Note that the base case $m = 0$ follows from \cite{S1997}*{Theorem 5.2}.
Fix $\chi : \{1, \ldots, n\} \to \{\ell, r\}$ and $\epsilon : \{1,\ldots, n\} \to \{1,2\}$ with $|\chi^{-1}(\{r\})|= m+1$.  

If $\chi(n) = r$, let $\chi' : \{1,\ldots, n\} \to \{\ell, r\}$ be defined by $\chi'(n) = \ell$ and $\chi'(y) = \chi(y)$ for all $y \neq n$.  Since
\[
R_k I_{N, \Omega} = L_k I_{N, \Omega}, \quad R^*_k I_{N, \Omega} = L^*_k I_{N, \Omega}, \qand R([a_{i,j}]) I_{N, \Omega} = L([a_{i,j}]) I_{N, \Omega},
\]
by replacing $Z^{\chi(n), \epsilon(n)}_{k_n}$ with $Z^{\chi'(n), \epsilon(n)}_{k_n}$, we obtain
\begin{align*}
\Phi\left(Z^{\chi(1), \epsilon(1)}_{k_1} \cdots Z^{\chi(n), \epsilon(n)}_{k_n}\right) &= \Phi\left(Z^{\chi'(1), \epsilon(1)}_{k_1} \cdots Z^{\chi'(n), \epsilon(n)}_{k_n} \right) \text{ and}\\
\Phi_\pi\left(Z^{\chi(1), \epsilon(1)}_{k_1}, \ldots, Z^{\chi(n), \epsilon(n)}_{k_n}\right) &= \Phi_\pi\left(Z^{\chi'(1), \epsilon(1)}_{k_1}, \ldots, Z^{\chi'(n), \epsilon(n)}_{k_n} \right),
\end{align*}
where $\pi \in BNC(\chi)$ is automatically an element of $BNC(\chi')$ as moving the bottom node from the right side to the left side is a bijection from $BNC(\chi)$ to $BNC(\chi')$.  Since this bijection preserves the coefficient 
\[
\sum_{\substack{\sigma \in BNC(\chi)  \\ \pi \leq \sigma \leq \epsilon }} \mu_{BNC}(\pi, \sigma),
\]
due to the relation between $\mu_{BNC}$ and $\mu_{NC}$ (see comment after Definition \ref{defn:mobius}), we obtain both sides of equation (\ref{eq:up-to-verify}) are preserved under this operation.  Consequently, equation (\ref{eq:up-to-verify}) holds in this case by the inductive hypothesis.

Otherwise, we may select $x \in \{1,\ldots, n\}$ such that $\chi(x) = r$ yet $\chi(y) = \ell$ for all $y > x$.  Let $\chi' : \{1,\ldots, n\} \to \{\ell, r\}$ be defined by $\chi'(x) = \ell$, $\chi'(x+1) = r$, and $\chi'(y) = \chi(y)$ for all $y \neq x,x+1$.  Similarly, let $\epsilon' : \{1,\ldots, n\} \to \{1,2\}$ be defined by $\epsilon'(x) = \epsilon(x+1)$, $\epsilon'(x+1) = \epsilon(x)$, and $\epsilon'(y) = \epsilon(y)$ for all $y \neq x,x+1$, and let
\[
Z^{\chi'(x), \epsilon'(x)}_{k'_x} = Z^{\chi(x+1), \epsilon(x+1)}_{k_{x+1}}, \quad Z^{\chi'(x+1), \epsilon'(x+1)}_{k'_{x+1}} = Z^{\chi(x), \epsilon(x)}_{k_{x}}, \qand Z^{\chi'(y), \epsilon'(y)}_{k'_y} = Z^{\chi(y), \epsilon(y)}_{k_{y}}
\]
for all $y \neq x, x+1$.  Note there is a bijection from $BNC(\chi)$ to $BNC(\chi')$ obtained by sending an element $\pi \in BNC(\chi)$ to $\pi' \in BNC(\chi')$ where $\pi'$ is obtained from $\pi$ by interchanging $x$ and $x+1$.

If 
\[
\left[ Z^{\chi(x), \epsilon(x)}_{k_{x}}, Z^{\chi(x+1), \epsilon(x+1)}_{k_{x+1}} \right] = 0
\]
then for all $\pi\in BNC(\chi)$
\begin{align*}
\Phi\left(Z^{\chi(1), \epsilon(1)}_{k_1} \cdots Z^{\chi(n), \epsilon(n)}_{k_n}\right) &= \Phi\left(Z^{\chi'(1), \epsilon'(1)}_{k'_1} \cdots Z^{\chi'(n), \epsilon'(n)}_{k'_n}\right) \text{ and} \\
\Phi_\pi\left(Z^{\chi(1), \epsilon(1)}_{k_1}, \ldots, Z^{\chi(n), \epsilon(n)}_{k_n}\right) &= \Phi_{\pi'}\left(Z^{\chi'(1), \epsilon'(1)}_{k'_1}, \ldots, Z^{\chi'(n), \epsilon'(n)}_{k'_n}\right)
\end{align*}
where the second equality holds for all $\pi$ as either $x$ and $x+1$ are in the same block of $\pi$ and the corresponding operators commute, or are in different blocks in which case the equality is trivial.  Since 
\[
\sum_{\substack{\sigma \in BNC(\chi)  \\ \pi \leq \sigma \leq \epsilon }} \mu_{BNC}(\pi, \sigma) = \sum_{\substack{\sigma' \in BNC(\chi')  \\ \pi' \leq \sigma' \leq \epsilon' }} \mu_{BNC}(\pi', \sigma')
\]
due to the connection between $\mu_{BNC}$ and $\mu_{NC}$, both sides of equation (\ref{eq:up-to-verify}) are preserved under this operation in this case.  Otherwise if
\[
\left[ Z^{\chi(x), \epsilon(x)}_{k_{x}}, Z^{\chi(x+1), \epsilon(x+1)}_{k_{x+1}} \right] \neq 0
\]
then we must be in the case
\[
Z^{\chi(x), \epsilon(x)}_{k_{x}} = R_k \qqand Z^{\chi(x+1), \epsilon(x+1)}_{k_{x+1}} = L^*_k
\]
or the case
\[
Z^{\chi(x), \epsilon(x)}_{k_{x}} = R^*_k \qqand Z^{\chi(x+1), \epsilon(x+1)}_{k_{x+1}} = L_k
\]
for some $k \in K$.  In the first case, we obtain that
\begin{align*}
\Phi\left(Z^{\chi(1), \epsilon(1)}_{k_1} \cdots Z^{\chi(n), \epsilon(n)}_{k_n}\right) 
&= \Phi\left(Z^{\chi'(1), \epsilon'(1)}_{k'_1} \cdots Z^{\chi'(n), \epsilon'(n)}_{k'_n}\right)  \\
& \quad + \Phi\left(Z^{\chi(1), \epsilon(1)}_{k_1} \cdots Z^{\chi(x-1), \epsilon(x-1)}_{k_{x-1}} I_{\X_N}\right) \Phi\left( I_{\X_N} Z^{\chi(x+2), \epsilon(x+2)}_{k_{x+2}} \cdots Z^{\chi(n), \epsilon(n)}_{k_n}\right)
\end{align*}
by parts (\ref{part:mm-1}) and (\ref{part:mm-3}) of Theorem \ref{thm:matrix-*-distributions}.  If $x \nsim_\pi x+1$, then we trivially obtain that
\[
\Phi_\pi\left(Z^{\chi(1), \epsilon(1)}_{k_1}, \ldots, Z^{\chi(n), \epsilon(n)}_{k_n}\right) = \Phi_{\pi'}\left(Z^{\chi'(1), \epsilon'(1)}_{k'_1}, \ldots, Z^{\chi'(n), \epsilon'(n)}_{k'_n}\right).
\]
However, if $x \sim_\pi x+1$, then 
\begin{align*}
&\Phi_\pi \left(Z^{\chi(1), \epsilon(1)}_{k_1} , \ldots, Z^{\chi(n), \epsilon(n)}_{k_n}\right)\\
 &= \Phi_{\pi'}\left(Z^{\chi'(1), \epsilon'(1)}_{k'_1}, \ldots, Z^{\chi'(n), \epsilon'(n)}_{k'_n}\right)  \\
& \qquad + \Phi_{\pi_1}\left(Z^{\chi(1), \epsilon(1)}_{k_1} , \ldots, Z^{\chi(x-1), \epsilon(x-1)}_{k_{x-1}}, I_{\X_N}\right) \Phi_{\pi_2}\left(I_{\X_N}, Z^{\chi(x+2), \epsilon(x+2)}_{k_{x+2}} , \ldots, Z^{\chi(n), \epsilon(n)}_{k_n}\right)
\end{align*}
where $\pi_1 = \pi|_{\{1,\ldots, x\}}$ (remove all of $x+1, \ldots, n$ otherwise keeping the blocks the same) and $\pi_2 = \pi|_{\{x+1,\ldots, n\}}$.
Note the map taking $\pi \in BNC(\chi)$ with $x \sim_\pi x+1$ to $(\pi_1, \pi_2) \in BNC(\chi|_{\{1,\ldots, x\}}) \times BNC(\chi|_{\{x+1, \ldots, n\}})$ is a bijection such that 
\[
\sum_{\substack{\sigma \in BNC(\chi)  \\ \pi \leq \sigma \leq \epsilon }} \mu_{BNC}(\pi, \sigma) = \left[ \sum_{\substack{\sigma_1 \in BNC(\chi|_{\{1,\ldots, x\}})  \\ \pi_1 \leq \sigma_1 \leq \epsilon|_{\{1,\ldots, x\}}}} \mu_{BNC}(\pi_1, \sigma_1)  \right]\left[ \sum_{\substack{\sigma_2 \in BNC(\chi|_{\{x+1,\ldots, n\}})  \\ \pi_2 \leq \sigma_2 \leq \epsilon|_{\{x+1,\ldots, n\}}}} \mu_{BNC}(\pi_2, \sigma_2)  \right]
\]
where $\sigma_1 = \sigma|_{\{1,\ldots, x\}}$  and $\sigma_2 = \sigma|_{\{x+1,\ldots, n\}}$.  Indeed the above equality can immediately be obtained using Definition \ref{defn:mobius} (which shows $\mu_{BNC}$ is completely determined by the lattice structure) followed by the multiplicative properties of the M\"{o}bius function to obtain 
\[
\mu_{BNC}(\pi, \sigma) = \mu_{BNC}(\pi_1 \cup \pi_2, \sigma_1 \cup \sigma_2) = \mu_{BNC}(\pi_1, \sigma_1) \mu_{BNC}(\pi_2, \sigma_2).
\]
Since, the inductive hypothesis implies
\begin{align*}
\Phi&\left(Z^{\chi(1), \epsilon(1)}_{k_1} \cdots Z^{\chi(x-1), \epsilon(x-1)}_{k_{x-1}} I_{\X_N}\right) \Phi\left( I_{\X_N} Z^{\chi(x+2), \epsilon(x+2)}_{k_{x+2}} \cdots Z^{\chi(n), \epsilon(n)}_{k_n}\right) \\
&= \sum_{\substack{ \pi_1 \in BNC(\chi|_{\{1,\ldots, x\}}) \\ \pi_2 \in BNC(\chi|_{\{x+1, \ldots, n\}})   }} \left[ \sum_{\substack{\sigma_1 \in BNC(\chi|_{\{1,\ldots, x\}})  \\ \pi_1 \leq \sigma_1 \leq \epsilon|_{\{1,\ldots, x\}}}} \mu_{BNC}(\pi_1, \sigma_1)  \right]\left[ \sum_{\substack{\sigma_2 \in BNC(\chi|_{\{x+1,\ldots, n\}})  \\ \pi_2 \leq \sigma_2 \leq \epsilon|_{\{x+1,\ldots, n\}}}} \mu_{BNC}(\pi_2, \sigma_2)  \right] \cdot\\
& \qquad \Phi_{\pi_1}\left(Z^{\chi(1), \epsilon(1)}_{k_1}, \ldots, Z^{\chi(x-1), \epsilon(x-1)}_{k_{x-1}}, I_{\X_N}\right) \Phi_{\pi_2}\left(I_{\X_N}, Z^{\chi(x+2), \epsilon(x+2)}_{k_{x+2}}, \ldots, Z^{\chi(n), \epsilon(n)}_{k_n}\right),
\end{align*}
we obtain that equation (\ref{eq:up-to-verify}) holds for the sequence $\left(Z^{\chi(1), \epsilon(1)}_{k_1}, \ldots, Z^{\chi(n), \epsilon(n)}_{k_n}\right) $ if and only if equation (\ref{eq:up-to-verify}) holds for the sequence $\left(Z^{\chi'(1), \epsilon'(1)}_{k'_1} \cdots Z^{\chi'(n), \epsilon'(n)}_{k'_n}\right)$ provided
\[
Z^{\chi(x), \epsilon(x)}_{k_{x}} = R_k \qqand Z^{\chi(x+1), \epsilon(x+1)}_{k_{x+1}} = L^*_k.
\]
The same holds by similar arguments when 
\[
Z^{\chi(x), \epsilon(x)}_{k_{x}} = R^*_k \qqand Z^{\chi(x+1), \epsilon(x+1)}_{k_{x+1}} = L_k.
\]

By repeating these arguments of interchanging left and right operators (which preserve both sides of equation (\ref{eq:up-to-verify})) we eventually reach the case that $\chi(n) = r$.  As a previous case implies equation (\ref{eq:up-to-verify}) holds when $\chi(n) = r$, equation (\ref{eq:up-to-verify})  holds.   Thus the result follows by the Principle of Mathematical Induction.
\end{proof}

\section{Bi-Matrix Models with $q$-Deformed Fock Space Entries}
\label{sec:q-deformed}

In this section, we will demonstrate the bi-free analogue of \cite{S1995}*{Theorem 2.1} by determining the asymptotic distributions of left and right matrices containing creation and annihilation operators on $q$-deformed Fock spaces.  Although this section is long and technical, it can be summarized via two points:
\begin{enumerate}
\item If $\{h_k\}_{k \in K}$ is an orthonormal set in a Hilbert space $\H$, then the two-faced families 
\[
\{(\{l(h_k), l^*(h_k)\}, \{r(h_k), r^*(h_k)\})\}_{k \in K}
\]
are bi-free and a bi-free central limit distribution by \cite{V2014}*{Theorem 7.4}.  Thus the joint distributions are completely determined by bi-non-crossing pair partitions where the $*$-terms precede the non-$*$-terms and the pairs must be of the same $k \in K$ colour.
\item The proof of \cite{S1995}*{Theorem 2.1} directly generalizes once the permutation $s_\chi$ is applied in the appropriate manner.
\end{enumerate}
As a consequence of Theorem \ref{thm:q-deformed}, Theorem \ref{thm:ass-bi-free} is automatically obtained via the $q= 1$ case and many implications from \cite{S1995} immediately carry forward to the bi-free setting.

We begin with definitions.
Let $\H$ be a (finite dimensional) Hilbert space.  Recall, for $q \in [-1,1]$, the $q$-deformed Fock space $\F_q(\H)$ is the Hilbert space
\[
\F_q(\H) = \bC \Omega_q \oplus \left(\bigoplus_{n \geq 1} \H^{\otimes n} \right)
\]
equipped with the inner product
\[
\langle g_1 \otimes \cdots \otimes g_n, h_1 \otimes \cdots \otimes h_m\rangle_q = \delta_{n,m} \sum_{\sigma \in S_n} q^{\mathrm{inv}(\sigma)} \prod^N_{k=1} \langle g_k, h_{\sigma(k)}\rangle_\H
\]
where $S_n$ is the permutation group on $\{1,\ldots, n\}$ and 
\[
\mathrm{inv}(\sigma) = |\{(i,j) \in \{1,\ldots, n\}^2 \, \mid \, i < j, \sigma(i) > \sigma(j)\} |
\]
is the number of inversions in $\sigma$.  We define the vacuum state $\varphi_q : \L(\F_q(\H)) \to \bC$ by
\[
\varphi_q(T) = \langle T \Omega_q, \Omega_q\rangle_q.
\]
For this section $\X$ will denote $\F_q(\H)$.

For each $h \in \H$, there are left (right) creation and annihilation operators, denoted $a(h)$ and $a^*(h)$ ($b(h)$ and $b^*(h)$) respectively, which are defined by
\begin{align*}
a(h) \Omega_q &= h, \qquad a^*(h) \Omega_q = 0, \\
b(h) \Omega_q &= h, \qquad b^*(h) \Omega_q = 0, \\
a(h) (h_1 \otimes \cdots \otimes h_n) &= h \otimes h_1 \otimes \cdots \otimes h_n, \\
b(h) (h_1 \otimes\cdots \otimes  h_n) &= h_1 \otimes \cdots \otimes h_n \otimes h, \\
a^*(h) (h_1 \otimes\cdots \otimes  h_n) &= \sum^n_{k=1} q^{k-1} \langle h_k, h\rangle h_1 \otimes \cdots \otimes \check{h_k} \otimes \cdots \otimes h_n, \text{ and }\\
b^*(h) (h_1 \otimes\cdots \otimes  h_n) &= \sum^n_{k=1} q^{n-k} \langle h_k, h\rangle h_1 \otimes \cdots \otimes \check{h_k}\otimes \cdots \otimes h_n,
\end{align*}
where $\check{h_k}$ denotes $h_k$ is omitted in the tensor.  Note $(a(h))^* = a^*(h)$ and $(b(h))^* = b^*(h)$. Furthermore, it is possible to verify the relations
\begin{align}
[a(h_1), b(h_2)] &= 0,\label{eq:q-comm-1} \\
[a^*(h_1), b^*(h_2)] &= 0, \label{eq:q-comm-2} \\
[a^*(h_1), b(h_2)] &= \langle h_2, h_1\rangle \left( \sum_{n\geq 0} q^n P_n\right), \label{eq:q-comm-3}\\
[b^*(h_1), a(h_2)] &= \langle h_2, h_1\rangle \left( \sum_{n\geq 0} q^n P_n\right), \label{eq:q-comm-4} \\
a^*(h_1)a(h_2) - q a(h_2) a^*(h_1) &= \langle h_2, h_1\rangle I_{\F_q(\H)}, \text{ and } \label{eq:q-comm-5} \\
b^*(h_1)b(h_2) - q b(h_2) b^*(h_1) &= \langle h_2, h_1\rangle I_{\F_q(\H)}, \label{eq:q-comm-6}
\end{align}
where $P_n \in \L(\X)$ is the projection onto the tensors of length $n$.

Let $N \in \bN$ and let $K$ be a fixed set.  Suppose $\left\{h_{i,j}^k \, \mid \, i,j \in \{1,\ldots, N\}, k \in K\right\}$ is an orthonormal set in $\H$.  For $\theta \in \{\ell, r\}$, define
\begin{align*}
Z^{\theta, \cdot}_k(N) &= \frac{1}{\sqrt{N}} Z^\theta\left(\left[z^\theta(h^k_{i,j})\right]\right) \qquad
&Z^{\theta, *}_k(N) &= \frac{1}{\sqrt{N}} Z^\theta\left(\left[\left(z^\theta\right)^*\left(h^k_{j,k}\right)\right]\right) \\
Z^{\theta, t}_k(N) &= \frac{1}{\sqrt{N}} Z^\theta\left(\left[z^\theta(h^k_{j,i})\right]\right) \qquad
&Z^{\theta, t*}_k(N) &= \frac{1}{\sqrt{N}} Z^\theta\left(\left[\left(z^\theta\right)^*(h^k_{i,j})\right]\right) 
\end{align*}
where, symbolically, $Z^\ell = L$, $Z^r = R$, $z^\ell = a$, and $z^r = b$.

\begin{thm}
\label{thm:q-deformed}
The asymptotic distribution of 
\[
\left\{Z^{\theta, \cdot}_k(N), Z^{\theta, *}_k(N), Z^{\theta, t}_k(N), Z^{\theta, t*}_k(N) \, \mid \, k \in K, \theta \in \{\ell, r\}\right\}
\]
with respect to $\Psi_N := \frac{1}{N} \tr \circ E_{\L(\X_N)}$ is equal to the distribution of 
\[
\left\{z^{\theta, \cdot}_k, z^{\theta, *}_k, z^{\theta, t}_k, z^{\theta, t*}_k \, \mid \, k \in K, \theta \in \{\ell, r\}\right\}
\]
with respect to $\varphi_0$, where $z_0^\ell = l$, $z_0^r = r$, $z^{\theta, \cdot}_k = z^\theta_0(h_k)$, $z^{\theta, *}_k = (z^\theta_0)^*(h_k)$, $z^{\theta, t}_k = z_0^\theta (h^t_k)$, $z^{\theta, t*}_k = (z_0^\theta)^* (h^t_k)$, and $\{h_k, h^t_k \, \mid \, k \in K\}$ is an orthonormal set.  Furthermore, the difference in the distributions of a fixed word is $O(\frac{1}{N})$ as $N \to \infty$.
\end{thm}
\begin{proof}
This proof is obtained by modifying the proof of \cite{S1995}*{Theorem 2.1}.  As such, many details are omitted by referring to \cite{S1995}.

It suffices to show that if $\chi : \{1,\ldots, n\} \to \{\ell, r\}$, $g : \{1,\ldots, n\} \to \{\cdot, *, t, t*\}$, and $k_1, \ldots k_n \in K$, then
\begin{align}
\lim_{N \to \infty} \Psi_N\left(Z^{\chi(1), g(1)}_{k_1}(N) \cdots Z^{\chi(n), g(n)}_{k_n}(N)\right) = \varphi_0\left( z^{\chi(1), g(1)}_{k_1} \cdots z^{\chi(n), g(n)}_{k_n}\right). \label{eq:to-show}
\end{align}
Recall 
\[
\left(  \left\{z^{\ell, \cdot}_k, z^{\ell, *}_k, z^{\ell, t}_k, z^{\ell, t*}_k \, \mid \, k \in K \right\}, \left\{z^{r, \cdot}_k, z^{r, *}_k, z^{r, t}_k, z^{r, t*}_k \, \mid \, k \in K \right\} \right)
\]
is a bi-free central limit distribution by \cite{V2014}*{Theorem 7.4} with
\[
\kappa\left(z^{\theta, w}_{k}, z^{\theta', w'}_{k'} \right) = 0
\]
unless $k = k'$ and either $w = *$ and $w' = \cdot$ or $w = t*$ and $w' = t$ (in which case it is equal to 1).   Consequently $\varphi_0 \left( z^{\chi(1), g(1)}_{k_1} \cdots z^{\chi(n), g(n)}_{k_n} \right) = 0$ unless there exists a $\pi\in BNC_2(\chi)$ such that 
\begin{enumerate}[(a)]
\item if $x \sim_\pi y$ then $k_x = k_y$, and \label{part:colour}
\item if $x \sim_\pi y$ with $x < y$, then either $g(x) = *$ and $g(y) = \cdot$, or $g(x) = t*$ and $g(y) = t$.  \label{part:orientation}
\end{enumerate} 
If such a $\pi$ exists, then $\varphi_0 \left( z^{\chi(1), g(1)}_{k_1} \cdots z^{\chi(n), g(n)}_{k_n}\right) = 1$ by the moment-cumulant formulae of Section \ref{sec:background}.

By Lemma \ref{lem:expanding-matrices} and Remark \ref{rem:indices-that-matter}, we obtain the left-hand-side of equation (\ref{eq:to-show}) is
\begin{align}
\frac{1}{N^{\frac{n}{2} + 1}} \sum^N_{j_1, \ldots, j_n = 1} \varphi_q\left(z_1\left(j_{s^{-1}_\chi(1)}, j_{s^{-1}_\chi(1) + 1}; k_1\right)z_2\left(j_{s^{-1}_\chi(2)}, j_{s^{-1}_\chi(2) + 1}; k_2\right) \cdots z_n\left(j_{s^{-1}_\chi(n)}, j_{s^{-1}_\chi(n) + 1}; k_n\right) \right) \label{eq:matrix-expanded}
\end{align}
where
\[
z_m(i, j; k) = \left\{
\begin{array}{ll}
a(h^k_{i,j}) & \mbox{if } g(m) = \cdot \\
a^*(h^k_{j,i}) & \mbox{if } g(m) = * \\
a(h^k_{j,i}) & \mbox{if } g(m) = t \\
a^*(h^k_{i,j}) & \mbox{if } g(m) = t*
\end{array} \right. \text{ when }\chi(m) = \ell
\]
and
\[
z_m(i, j; k) = \left\{
\begin{array}{ll}
b(h^k_{i,j}) & \mbox{if } g(m) = \cdot \\
b^*(h^k_{j,i}) & \mbox{if } g(m) = * \\
b(h^k_{j,i}) & \mbox{if } g(m) = t \\
b^*(h^k_{i,j}) & \mbox{if } g(m) = t*
\end{array} \right. \text{ when }\chi(m) = r.
\]

Notice that if $h^k_{i,j}$ is replaced by $z h^k_{i,j}$ for some $z \in \bC$ with $|z| = 1$, the distribution of any term in expression (\ref{eq:matrix-expanded}) is not changed.  Therefore, since $h \mapsto a(h)$ and $h \mapsto b(h)$ are linear in $h$, we obtain that 
\[
\varphi_q\left(z_1\left(j_{s^{-1}_\chi(1)}, j_{s^{-1}_\chi(1) + 1}; k_1\right) \cdots z_n\left(j_{s^{-1}_\chi(n)}, j_{s^{-1}_\chi(n) + 1}; k_n\right) \right)= 0
\]
unless there is a $\pi \in \P_2(n)$ such that 
\begin{enumerate}
\item $x \sim_\pi y$ implies $k_x = k_y$, and \label{part:colour2}
\item $x \sim_\pi y$ and $x \neq y$ implies that the set $\{g(x), g(y)\}$ is one $\{\ast, \cdot\}, \{\ast, t\}, \{t \ast, \cdot\}, \{t \ast, t\}$.\label{part:orientation2}
\end{enumerate}
Moreover, for such a partition $\pi$, we have that if $x \sim_\pi y$ and $x \neq y$ then
\begin{align}
j_{s^{-1}_\chi(x)} = j_{s^{-1}_\chi(y) + 1} \text{ and } j_{s^{-1}_\chi(x)+1} = j_{s^{-1}_\chi(y)} & \quad \text{if } \{g(x), g(y)\} = \{\ast, \cdot\} \text{ or }\{t\ast, t\} \label{part:first-j-relation}  \\
j_{s^{-1}_\chi(x)} = j_{s^{-1}_\chi(y)} \text{ and } j_{s^{-1}_\chi(x)+1} = j_{s^{-1}_\chi(y) + 1} & \quad \text{otherwise}. \label{part:second-j-relation}
\end{align}

Denote by $\Theta$ the set of all partitions satisfying (\ref{part:colour2}) and (\ref{part:orientation2}).  Then expression (\ref{eq:matrix-expanded}) equals
\[
\frac{1}{N^{\frac{n}{2} + 1}} \sum^N_{\substack{ j_1, \ldots, j_s = 1   \\ j_k \text{ satisfy } (\ref{part:first-j-relation}), (\ref{part:second-j-relation}) \text{ for some } \pi \in \Theta  }} \varphi_q\left(z_1\left(j_{s^{-1}_\chi(1)}, j_{s^{-1}_\chi(1) + 1}; k_1\right) \cdots z_n\left(j_{s^{-1}_\chi(n)}, j_{s^{-1}_\chi(n) + 1}; k_n\right) \right).
\]
The sum is zero if $\Theta$ is empty, which occurs when $n$ is odd.  Since $\varphi_0\left( z^{\chi(1), g(1)}_{k_1} \cdots z^{\chi(n), g(n)}_{k_n}\right) = 0$  by above discussions when $n$ is odd, we may assume $n$ is even for the remainder of the proof.

As in \cite{S1995}, let $C$ be a polygon with $n$ edges and let 
\[
D = \{\text{maps from the vertices on }C\text{ to }\{1,\ldots, N\}\}.
\]
Enumerate the vertices of $C$ clockwise.  Then the tuples $\{(j_1, \ldots, j_n) \, \mid \, j_k \in \{1,\ldots, N\}\}$ can be identified with elements of $D$ in the trivial way where $j_k$ is the value of an element of $D$ on the $k^{\mathrm{th}}$ vertex.  Orient the edges of $C$ via $1 \to 2 \to \cdots \to n \to 1$ and label the edge from $k$ to $k+1$ (where $n+1 = 1$) with $s_\chi(k)$.  Thus $\pi$ can be viewed as a partition on the set of edges of $C$ labelled in this way.  Consider the quotient graph $C/\pi$, where two edges $x$ and $y$ with $x \neq y$ and $x \sim_\pi y$ are ``glued'' together with orientation reversion if $\{g(x), g(y) \} = \{\ast, \cdot\}$ or $\{t\ast, t\}$ and with orientation preservation otherwise.  Hence if $s_\chi(x) \sim_\pi s_\chi(y)$ then the vertices $j_x$ and $j_{y+1}$ pair up and the vertices $j_{x+1}$ and $j_{y}$ pair up  if orientation is reversed and the vertices $j_x$ and $j_{y}$ pair up and the vertices $j_{x+1}$ and $j_{y+1}$ pair up  if orientation is preserved.  Note there is a clear injection $i$ from 
\[
D_\pi :=\{\text{maps from the vertices on }C/\pi\text{ to }\{1,\ldots, N\}\}
\]
into $D$ such that $(j_1, \ldots, j_n)$ satisfies (\ref{part:first-j-relation}) and (\ref{part:second-j-relation}) if and only if $(j_1, \ldots, j_n)$ lies in the image of $i : D_\pi \to D$.
Using this notation, expression (\ref{eq:matrix-expanded}) now becomes
\[
\frac{1}{N^{\frac{n}{2} + 1}} \sum_{\substack{ (j_1, \ldots, j_n) \in i(D_\pi) \\ \text{for some }\pi \in \Theta}} \varphi_q\left(z_1\left(j_{s^{-1}_\chi(1)}, j_{s^{-1}_\chi(1) + 1}; k_1\right) \cdots z_n\left(j_{s^{-1}_\chi(n)}, j_{s^{-1}_\chi(n) + 1}; k_n\right) \right).
\]

As in \cite{S1995}, the number of elements of $D_\pi$ that are not injective is of order at most $O(N^{\frac{n}{2}})$.  Therefore, if $D'_\pi$ is the subset of $D_\pi$ consisting of injective functions then, modulo terms of $O(\frac{1}{N})$, the expression (\ref{eq:matrix-expanded}) equals
\[
\frac{1}{N^{\frac{n}{2} + 1}} \sum_{\substack{ (j_1, \ldots, j_n) \in i(D'_\pi) \\ \text{for some }\pi \in \Theta}} \varphi_q\left(z_1\left(j_{s^{-1}_\chi(1)}, j_{s^{-1}_\chi(1) + 1}; k_1\right) \cdots z_n\left(j_{s^{-1}_\chi(n)}, j_{s^{-1}_\chi(n) + 1}; k_n\right) \right).
\]

If $\pi$ satisfies (\ref{part:colour2}) and (\ref{part:orientation2}) and $C/\pi$ has at most $\frac{n}{2}$ vertices, then $|D_\pi| = O(N^\frac{n}{2})$ so such $\pi$ may be ignored asymptotically in the sum.  As $C$ has $n$ edges, $C/\pi$ has at most $\frac{n}{2} + 1$ vertices.  By \cite{S1995}*{Lemmata 2.2, 2.3}, such a $C/\pi$ has precisely $\frac{n}{2} + 1$ vertices if and only if $\pi$ non-crossing on $(s_\chi(1), \ldots, s_\chi(n))$ (with the nodes in that order) and the edges of $C$ were glued together only using orientation reversion.  Hence $\pi \in BNC(\chi)$ and we need only consider $\pi$ satisfying (\ref{part:first-j-relation}) everywhere.

Let $\Theta'$ denote the subset of $\Theta$ consisting of $\pi \in BNC(\chi)$ satisfying (\ref{part:first-j-relation}) everywhere.  It is possible to verify that $|i(D'_\pi) \cap i(D'_{\pi'})| = O(N^\frac{n}{2})$ whenever $\pi, \pi' \in \Theta'$ are such that $\pi \neq \pi'$.  Thus, up to $O(\frac{1}{N})$, expression (\ref{eq:matrix-expanded}) equals
\[
\sum_{\pi \in \Theta'} \frac{1}{N^{\frac{n}{2} + 1}} \sum_{(j_1, \ldots, j_n) \in i(D'_\pi)} \varphi_q\left(z_1\left(j_{s^{-1}_\chi(1)}, j_{s^{-1}_\chi(1) + 1}; k_1\right) \cdots z_n\left(j_{s^{-1}_\chi(n)}, j_{s^{-1}_\chi(n) + 1}; k_n\right) \right).
\]
Again, an empty sum is zero by definition.

If $\pi \in \Theta'$ is such that there exists $x \sim_\pi y$ with $x < y$ and $g(x) = \cdot$ (so $g(y) = \ast$), then we claim that
\[
\sum_{(j_1, \ldots, j_n) \in i(D'_\pi)} \varphi_q\left(z_1\left(j_{s^{-1}_\chi(1)}, j_{s^{-1}_\chi(1) + 1}; k_1\right) \cdots z_n\left(j_{s^{-1}_\chi(n)}, j_{s^{-1}_\chi(n) + 1}; k_n\right) \right) = 0.
\]
To see this, we recall that $D'_\pi$ consists of injective elements so the set $\left\{j_{s^{-1}_\chi(x)}, j_{s^{-1}_\chi(x)+1}\right\}$ is not equal to any set $\left\{j_{s^{-1}_\chi(m)}, j_{s^{-1}_\chi(m)+1}\right\}$ for $m \neq x,y$.  Since $z_y\left(j_{s^{-1}_\chi(y)}, j_{s^{-1}_\chi(y) + 1}; k_y\right)$ is either $a^*(h)$ or $b^*(h)$ for some $h \in \H$, we obtain using equations (\ref{eq:q-comm-2}), (\ref{eq:q-comm-3}), (\ref{eq:q-comm-4}), (\ref{eq:q-comm-5}), and (\ref{eq:q-comm-6}) that
\[
\varphi_q\left(z_1\left(j_{s^{-1}_\chi(1)}, j_{s^{-1}_\chi(1) + 1}; k_1\right) \cdots z_n\left(j_{s^{-1}_\chi(n)}, j_{s^{-1}_\chi(n) + 1}; k_n\right) \right) = \pm q^m\varphi_q\left(\ldots z_y\left(j_{s^{-1}_\chi(y)}, j_{s^{-1}_\chi(y) + 1}; k_y\right)\right)
\]
for some $m \in \bN \cup \{0\}$.  Since $z_y\left(j_{s^{-1}_\chi(y)}, j_{s^{-1}_\chi(y) + 1}; k_y\right)\Omega_q = 0$, the claim follows.  Similarly if $\pi \in \Theta'$ is such that there exists $x \sim_\pi y$ with $x < y$ and $g(x) = t$ (so $g(y) = t\ast$), then
\[
\sum_{(j_1, \ldots, j_n) \in i(D'_\pi)} \varphi_q\left(z_1\left(j_{s^{-1}_\chi(1)}, j_{s^{-1}_\chi(1) + 1}; k_1\right) \cdots z_n\left(j_{s^{-1}_\chi(n)}, j_{s^{-1}_\chi(n) + 1}; k_n\right) \right) = 0.
\]
Thus we need only consider $\pi \in \Theta'$ satisfying (\ref{part:orientation}).

Suppose $\pi \in \Theta'$ satisfies (\ref{part:orientation}).  We claim that 
\[
\varphi_q\left(z_1\left(j_{s^{-1}_\chi(1)}, j_{s^{-1}_\chi(1) + 1}; k_1\right) \cdots z_n\left(j_{s^{-1}_\chi(n)}, j_{s^{-1}_\chi(n) + 1}; k_n\right) \right) = 1.
\]
Indeed, using just equations (\ref{eq:q-comm-1}), (\ref{eq:q-comm-2}), (\ref{eq:q-comm-3}), and (\ref{eq:q-comm-4})
one can write 
\[
z_1\left(j_{s^{-1}_\chi(1)}, j_{s^{-1}_\chi(1) + 1}; k_1\right) \cdots z_n\left(j_{s^{-1}_\chi(n)}, j_{s^{-1}_\chi(n) + 1}; k_n\right)
\]
as a product of
\begin{enumerate}[i)]
\item \[
a^{g'(1)}(h'_1) \cdots a^{g'(m)}(h'_m)
\]
where $\{h'_1, \ldots, h'_m\}$ are unit vectors, $g' : \{1,\ldots, m\} \to \{\cdot, \ast\}$, and there is a pair non-crossing partition $\pi'$ on $\{1,\ldots, m\}$ such that if $x \sim_\pi y$ and $x < y$ then $g'(x) = \ast$, $g'(y) = \cdot$, and $h'_x = h'_y$, and $h'_x \bot h'_y$ if $x \nsim_\pi y$,
\item \[
b^{g'(1)}(h'_1) \cdots b^{g'(m)}(h'_m)
\]
where $\{h'_1, \ldots, h'_m\}$ are unit vectors, $g' : \{1,\ldots, m\} \to \{\cdot, \ast\}$, and there is a pair non-crossing partition $\pi'$ on $\{1,\ldots, m\}$ such that if $x \sim_\pi y$ and $x < y$ then $g'(x) = \ast$, $g'(y) = \cdot$, and $h'_x = h'_y$, and $h'_x \bot h'_y$ if $x \nsim_\pi y$, and
\item 
\[
b^*(h) a(h) \qqand a^*(h) b(h)
\]
where $h$ is a unit vector.
\end{enumerate}
This is best seen via the following process on bi-non-crossing diagrams (thinking of equations (\ref{eq:q-comm-1}), (\ref{eq:q-comm-2}), (\ref{eq:q-comm-3}), and (\ref{eq:q-comm-4})) where $h_x \bot h_y$ if $x \neq y$.
\begin{align*}
\begin{tikzpicture}[baseline]
	\draw[thick, dashed] (-1,7.75) -- (-1,-.25) -- (1,-.25) -- (1,7.75);
	\draw[thick,black] (1,0) -- (0,0) -- (0, 1.5) -- (-1, 1.5);
	\draw[thick,black] (1,5.5) -- (0,5.5) -- (0, 3) -- (-1, 3);
	\draw[thick, black] (-1, 7.5) -- (-0.333, 7.5) -- (-0.333, 4) -- (-1, 4);
		\draw[thick, black] (-1, 6.5) -- (-0.667, 6.5) -- (-0.667, 4.5) -- (-1, 4.5);
			\draw[thick, black] (-1, 2.5) -- (-0.333, 2.5) -- (-0.333, 2) -- (-1, 2);
	\draw[thick, black] (1, 0.5) -- (0.333, .5) -- (.333, 5) -- (1,5);
	\draw[thick, black] (1, 6) -- (0.333, 6) -- (.333, 7) -- (1,7);
	\draw[thick, black] (1, 3.5) -- (0.667, 3.5) -- (.667, 1) -- (1,1);
	\node[left] at (-1, 7.5) {$a^*(h_6)$};
	\draw[black, fill=black] (-1,7.5) circle (0.05);	
		\node[right] at (1, 7) {$b^*(h_8)$};
	\draw[black, fill=black] (1,7) circle (0.05);
		\node[left] at (-1, 6.5) {$a^*(h_7)$};
	\draw[black, fill=black] (-1,6.5) circle (0.05);
		\node[right] at (1, 6) {$b(h_8)$};
	\draw[black, fill=black] (1,6) circle (0.05);
		\node[right] at (1, 5.5) {$b^*(h_5)$};
	\draw[black, fill=black] (1,5.5) circle (0.05);
		\node[left] at (-1, 4.5) {$a(h_7)$};
	\draw[black, fill=black] (-1,4.5) circle (0.05);
		\node[left] at (-1, 4) {$a(h_6)$};
	\draw[black, fill=black] (-1,4) circle (0.05);
		\node[right] at (1, 5) {$b^*(h_2)$};
	\draw[black, fill=black] (1,5) circle (0.05);
		\node[right] at (1,3.5) {$b^*(h_3)$};
	\draw[black, fill=black] (1,3.5) circle (0.05);
		\node[left] at (-1, 3) {$a(h_5)$};
	\draw[black, fill=black] (-1,3) circle (0.05);
		\node[left] at (-1, 2.5) {$a^*(h_4)$};
	\draw[black, fill=black] (-1,2.5) circle (0.05);
		\node[left] at (-1, 2) {$a(h_4)$};
	\draw[black, fill=black] (-1,2) circle (0.05);
		\node[left] at (-1, 1.5) {$a^*(h_1)$};
	\draw[black, fill=black] (-1,1.5) circle (0.05);
		\node[right] at (1, 1) {$b(h_3)$};
	\draw[black, fill=black] (1,1) circle (0.05);
		\node[right] at (1, .5) {$b(h_2)$};
	\draw[black, fill=black] (1,.5) circle (0.05);
		\node[right] at (1, 0) {$b(h_1)$};
	\draw[black, fill=black] (1,0) circle (0.05);
	\draw[thick, black] (2.5, 3.75) -- (4, 3.75) -- (3.95, 3.7);
	\draw[thick, black] (4, 3.75) -- (3.95, 3.8);
\end{tikzpicture}
\begin{tikzpicture}[baseline]
	\draw[thick, dashed] (-1,7.75) -- (-1,-.25) -- (1,-.25) -- (1,7.75);
	\draw[thick,black] (1,0) -- (0,0) -- (0, .5) -- (-1, .5);
	\draw[thick,black] (1,4.5) -- (0,4.5) -- (0, 4) -- (-1, 4);
	\draw[thick, black] (-1, 6.5) -- (-0.333, 6.5) -- (-0.333, 5) -- (-1, 5);
	\draw[thick, black] (-1, 6) -- (-0.667, 6) -- (-0.667, 5.5) -- (-1, 5.5);
		\draw[thick, black] (-1, 3.5) -- (-0.333, 3.5) -- (-0.333, 3) -- (-1, 3);
	\draw[thick, black] (1, 2.5) -- (0.333, 2.5) -- (.333, 1) -- (1,1);
	\draw[thick, black] (1, 7.5) -- (0.333, 7.5) -- (.333, 7) -- (1,7);
	\draw[thick, black] (1, 1.5) -- (0.667, 1.5) -- (.667, 2) -- (1,2);
	\node[left] at (-1, 6.5) {$a^*(h_6)$};
	\draw[black, fill=black] (-1,6.5) circle (0.05);	
		\node[right] at (1, 7.5) {$b^*(h_8)$};
	\draw[black, fill=black] (1,7.5) circle (0.05);
		\node[left] at (-1, 6) {$a^*(h_7)$};
	\draw[black, fill=black] (-1,6) circle (0.05);
		\node[right] at (1, 7) {$b(h_8)$};
	\draw[black, fill=black] (1,7) circle (0.05);
		\node[right] at (1, 4.5) {$b^*(h_5)$};
	\draw[black, fill=black] (1,4.5) circle (0.05);
		\node[left] at (-1, 5.5) {$a(h_7)$};
	\draw[black, fill=black] (-1,5.5) circle (0.05);
		\node[left] at (-1, 5) {$a(h_6)$};
	\draw[black, fill=black] (-1,5) circle (0.05);
		\node[right] at (1, 2.5) {$b^*(h_2)$};
	\draw[black, fill=black] (1,2.5) circle (0.05);
		\node[right] at (1,2) {$b^*(h_3)$};
	\draw[black, fill=black] (1,2) circle (0.05);
		\node[left] at (-1, 4) {$a(h_5)$};
	\draw[black, fill=black] (-1,4) circle (0.05);
		\node[left] at (-1, 3.5) {$a^*(h_4)$};
	\draw[black, fill=black] (-1,3.5) circle (0.05);
		\node[left] at (-1, 3) {$a(h_4)$};
	\draw[black, fill=black] (-1,3) circle (0.05);
		\node[left] at (-1, .5) {$a^*(h_1)$};
	\draw[black, fill=black] (-1,.5) circle (0.05);
		\node[right] at (1, 1.5) {$b(h_3)$};
	\draw[black, fill=black] (1,1.5) circle (0.05);
		\node[right] at (1, 1) {$b(h_2)$};
	\draw[black, fill=black] (1,1) circle (0.05);
		\node[right] at (1, 0) {$b(h_1)$};
	\draw[black, fill=black] (1,0) circle (0.05);
\end{tikzpicture}
\end{align*}
 Since each of the four possible products when applied to $\Omega_q$ produces $\Omega_q$, we obtain that
\[
z_1\left(j_{s^{-1}_\chi(1)}, j_{s^{-1}_\chi(1) + 1}; k_1\right) \cdots z_n\left(j_{s^{-1}_\chi(n)}, j_{s^{-1}_\chi(n) + 1}; k_n\right)\Omega_q = \Omega_q
\]
thereby completing the claim. Therefore
\[
\sum_{(j_1, \ldots, j_n) \in i(D'_\pi)} \varphi_q\left(z_1\left(j_{s^{-1}_\chi(1)}, j_{s^{-1}_\chi(1) + 1}; k_1\right) \cdots z_n\left(j_{s^{-1}_\chi(n)}, j_{s^{-1}_\chi(n) + 1}; k_n\right) \right) = |D'_\pi|
\]
whenever $\pi \in \Theta'$ satisfies (\ref{part:orientation}).

Hence, we have demonstrated that expression (\ref{eq:matrix-expanded}) is equal to
\[
\frac{1}{N^{\frac{n}{2} + 1}}\sum_{\substack{ \pi \in BNC_2(\chi) \\ \pi \text{ satisfies } (\ref{part:colour}) \text{ and }(\ref{part:orientation})}} |D'_\pi|.
\]
The first part of this proof shows, either there are no such $\pi$ (in which case we get both sides of equation (\ref{eq:to-show}) equal to zero), or there is a unique $\pi$ (in which case the right-hand-side of equation (\ref{eq:to-show}) was 1).  In the case of a unique $\pi$, since $|D'_\pi|N^{-\frac{n}{2} - 1} = 1 + O(\frac{1}{N})$ as $D'_\pi$ is the set of injective maps on  $N^{\frac{n}{2} + 1}$ vertices, we obtain the right-hand-side of equation (\ref{eq:to-show}) is also 1 thereby completing the proof.
\end{proof}

As Theorem \ref{thm:q-deformed} parallels \cite{S1995}*{Theorem 2.1}, we immediately obtain similar applications.  Indeed asymptotic bi-freeness from ``constant block matrix algebras" holds.  That is, for $n$ a positive integer divisor of $N$, let
\[
M_n(N) = \M_n(\bC) \otimes I_{\frac{N}{n}} \subseteq \M_n(\bC) \otimes \M_{\frac{N}{n}}(\bC) \subseteq \M_N(\bC).
\]
\begin{thm}
\label{thm:q-constant}
Let 
\[
\left\{Z^{\theta, \cdot}_k(N), Z^{\theta, *}_k(N), Z^{\theta, t}_k(N), Z^{\theta, t*}_k(N) \, \mid \, k \in K, \theta \in \{\ell, r\}\right\}
\]
be as in Theorem \ref{thm:q-deformed}, and let $n$ be a fixed positive integer.  For multiples $N$ of $n$, 
\[
\left( \left\{Z^{\ell, \cdot}_k(N), Z^{\ell, *}_k(N), Z^{\ell, t}_k(N), Z^{\ell, t*}_k(N)\right\}_{k \in K}, \left\{Z^{r, \cdot}_k(N), Z^{r, *}_k(N), Z^{r, t}_k(N), Z^{r, t*}_k(N) \right\}_{k \in K} \right)
\]
is asymptotically bi-free from $(L(M_n(N)), R(M_n(N)))$ with respect to $\Psi_N := \frac{1}{N} \tr \circ E_{\L(\X_N)}$ as $N \to \infty$.
\end{thm}
\begin{proof}
Each $Z^{\theta, w}_k(N)$ can be decomposed into an $n \times n$ block matrix with $\frac{N}{n} \times\frac{N}{n}$ matrix entries.  By Construction \ref{cons:twisting-action}, Remark \ref{rem:non-commutative-probability-space-canonical-action}, and Theorem \ref{thm:q-deformed}, each $\frac{N}{n} \times\frac{N}{n}$ tends to a left or right creation or annihilation operator.  As these left or right creation or annihilation operators appear in the correct entries to invoke Theorem \ref{thm:asymptotic-bi-free-fock}, the result follows.
\end{proof}
Similarly, if $\Delta(N) \subseteq \M_N(\bC)$ denotes the subalgebra of diagonal matrices and $\{D_k(N)\}_{k \in K} \subseteq \Delta(N)$ are diagonal matrices such that $\{D_k(N)\}_{k \in K}$ has a limit distribution as $N \to \infty$ and $\sup_{N \in \bN} \left\| D_k(N)\right\| < \infty$ for all $k \in K$, then the following result holds as each $D_k(N)$ can be approximated by elements of $\Delta(N) \cap Y_n(N)$ for sufficiently large $n$ and $N$.
\begin{thm}
\label{thm:q-diagonal}
Let 
\[
\left\{Z^{\theta, \cdot}_k(N), Z^{\theta, *}_k(N), Z^{\theta, t}_k(N), Z^{\theta, t*}_k(N) \, \mid \, k \in K, \theta \in \{\ell, r\}\right\}
\]
be as in Theorem \ref{thm:q-deformed}, and let $n$ be a fixed positive integer.  Then
\[
\left( \left\{Z^{\ell, \cdot}_k(N), Z^{\ell, *}_k(N), Z^{\ell, t}_k(N), Z^{\ell, t*}_k(N)\right\}_{k \in K}, \left\{Z^{r, \cdot}_k(N), Z^{r, *}_k(N), Z^{r, t}_k(N), Z^{r, t*}_k(N) \right\}_{k \in K} \right)
\]
is asymptotically bi-free from $(\{L(D_k(N))\}_{k \in K}, \{R(D_k(N))\}_{k \in K} )$ with respect to $\Psi_N := \frac{1}{N} \tr \circ E_{\L(\X_N)}$ as $N \to \infty$.
\end{thm}

\begin{rem}
\label{rem:q-matrix-results}
We note that Theorem \ref{thm:q-deformed} may be used to obtain many bi-free analogues of random matrix results as \cite{S1995}*{Section 3} does.  However, we will not formal state all such results.

For one example, if $q = 1$ (the bosonic case) and $\{h_k\}_{k \in K} \cup \{h'_k\}_{k \in K}$ is an orthonormal set, then, as in \cite{S1995}*{Section 3}, $a(h_k) + a^*(h_k)$ and $b(h_m) + b^*(h_m)$ are distributed with respect to $\varphi_q$ as Gaussian random variables of variance 1, commute with each other, and have a covariance matrix dependent on $\langle h_k, h_m\rangle$.  Similarly $a(h_k) + a^*(h'_k)$ and $b(h_m) + b^*(h'_m)$ are distributed with respect to $\varphi_q$ as centred complex Gaussian random variables of variance 1, commute with each other, and have a covariance matrix dependent on $\langle h_k, h'_m\rangle$ and $\langle h'_k, h_m\rangle$.  Consequently, by using  $Z^{\theta, \cdot}_k(N) + Z^{\theta, *}_k(N)$ for $\theta \in \{\ell, r\}$ and Theorem \ref{thm:q-deformed} which determines the limit distributions, several bi-free central limit distributions (contained in those where left operator commute with right operators in distribution) may be obtained.  Furthermore, Theorem \ref{thm:ass-bi-free} may be obtain using this method and Theorems \ref{thm:q-constant} and \ref{thm:q-diagonal} imply such pairs of faces are asymptotically bi-free from ``constant block matrix algebras'' and ``constant diagonal matrices''.  These later results also hold in the fermonic case ($q = -1$).

Similar results hold for ``real Gaussian random matrices" by using the real and imaginary parts of $Z^{\theta, \cdot}_k(N) + Z^{\theta, t*}_k(N)$ for $\theta \in \{\ell, r\}$ since $Z^{\theta, \cdot}_k(N) + Z^{\theta, t*}_k(N)$ converges in distribution to a free circular element acting on either the left or right so that the real and imaginary parts are free semicircular elements.
\end{rem}

\section{Bi-Matrix Models producing Boolean and Monotone Central Limit Distributions}
\label{sec:boolean-monotone}

In this section, we demonstrate the existence of left and right matrices of creation and annihilation operators on a Fock space that, when combined in certain patterns, have asymptotic distributions equal to Boolean independent, Boolean central limit distributions and monotonically independent, monotone central limit distributions.  The main ideas behind these models arise from \cite{S2014}, which demonstrates ways Boolean and monotone independence arise from bi-free pairs of faces.  One motivation for these results is an alternate model for such distributions than those provided in \cite{L2014}.

\subsection*{Boolean Bi-Matrix Models}  Recall, by \cite{SW1997}, an operator $Z$ in a non-commutative probability space $(\A, \varphi)$ is a \emph{Boolean central limit distribution} (with variance 1) if 
\[
\varphi(Z^m) = \int_{\bR} t^m \, d\mu_B
\]
where $\mu_B = \frac{1}{2}\left( \delta_{-1} + \delta_1\right)$ (the average of two point-masses; one at $-1$ and one at 1).  Furthermore, operators $Z_1, \ldots, Z_n \in \A$ are said to be \emph{Boolean independent with respect to $\varphi$} if
\[
\varphi\left(Z_{k_1}^{w_1} \cdots Z^{w_m}_{k_m}\right) = \prod^m_{j=1} \varphi\left(Z^{w_j}_{k_j}\right)
\]
whenever $m \in \bN$,  $w_1, \ldots, w_m \in \bN$, and $k_1, \ldots, k_m \in \{1,\ldots, n\}$ are such that $k_j \neq k_{j+1}$ for all $j$.

Let $K$ be an arbitrary set and let $\{h_{j,k} \, \mid \, j \geq 1, k \in K\}$ be an orthonormal basis for a Hilbert space $\H$.  If $\X = \F(\H)$, for each $N \in \bN$ and $k \in K$ let
\[
T_k(N) := \sum^{\lfloor \frac{N}{2} \rfloor}_{j=1} l^*(h_{j,k}) \otimes  E_{2j-1, 2j}(N) + \sum^{\lfloor \frac{N-1}{2} \rfloor  }_{j=1} l(h_{j,k}) \otimes E_{2j, 2j+1}(N).
\]
That is, $T_k$ is an upper triangular matrix with the sequence $l^*(h_{1,k}), l(h_{1,k}), l^*(h_{2,k}), l(h_{2,k}), \ldots$ on the sub-diagonal and zero elsewhere.  Furthermore, let
\[
S(N) := \sum^N_{j=1} E_{j, j-1}(N).
\]
For each $k \in K$ let $L_k(N) := L(T_k(N))$ and $R(N) := R(S(N))$. 
\begin{thm}
Using the above notation, for each $k \in K$
\[
\lim_{N \to \infty} \frac{2}{N} \tr\left( E_{\L(\X_N)}\left( \left(L_k(N) R(N) \right)^m  \right)   \right) = \int_{\bR} x^m \, d\mu_B
\]
for all $m \in \bN$.  
Furthermore, the operators $\{L_k(N) R(N)\}_{k \in K}$ are asymptotic Boolean independent with respect to $\frac{2}{N} \tr \circ E_{\L(\X_N)}$.
\end{thm}
\begin{proof}
The idea of the proof is simply to compute the diagonal entries of
\[
L_{\epsilon(1)}(N) R(N) L_{\epsilon(2)}(N) R(N) \cdots L_{\epsilon(n)}(N) R(N) I_{N, \Omega}
\]
for each $n \in \bN$ and  $\epsilon : \{1, \ldots, n\} \to K$.  Indeed we claim the $x^{\mathrm{th}}$ entry along the diagonal is
\[
\left\{
\begin{array}{ll}
l^*_{y, \epsilon(1)} l_{y, \epsilon(2)} l^*_{y+1, \epsilon(3)} l_{y+1, \epsilon(4)} \cdots & \mbox{if } x=2y-1 \leq N - n  \\
l_{y, \epsilon(1)} l^*_{y+1, \epsilon(2)} l_{y+1, \epsilon(3)}, l^*_{y+2, \epsilon(4)} \cdots & \mbox{if } x = 2y \leq N - n \\
0 & \mbox{otherwise}
\end{array} \right. .
\]
This can be verified by appealing to Remark \ref{rem:commutative} and a direct computation of
\[
T_{\epsilon(1)}(N)T_{\epsilon(2)}(N) \cdots T_{\epsilon(n)}(N) S(N)^n.
\]
One then sees that
\begin{align*}
& \frac{2}{N} \tr\left( E_N\left( L_{\epsilon(1)}(N) R(S(N)) L_{\epsilon(2)}(N) R(S(N)) \cdots L_{\epsilon(n)}(N) R(S(N)) \right) \right) \\
&= \frac{2}{N} \left\{
\begin{array}{ll}
\max\{0, \lfloor\frac{N-n}{2}\rfloor \} & \mbox{if } n \mbox{ is even and } \epsilon(2m-1) = \epsilon(2m) \mbox{ for all }m  \\
0 & \mbox{otherwise}
\end{array} \right. .
\end{align*}
As $\lim_{N \to \infty} \frac{2}{N}\max\{0, \lfloor\frac{N-n}{2}\rfloor \} = 1$, the result follows.
\end{proof}

\begin{rem}
Note one may instead use
\[
T_k(N) := \sum^{\lfloor \frac{N}{2} \rfloor}_{j=1} l^*(h_{1,k}) \otimes  E_{2j-1, 2j}(N) + \sum^{\lfloor \frac{N-1}{2} \rfloor  }_{j=1} l(h_{1,k}) \otimes  E_{2j, 2j+1}(N)
\]
and the above computations will still work (i.e. we only need $|K|$ creation/annihilation operators).
\end{rem}

\subsection*{Monotone Bi-Matrix Models} For the monotone setting, as the main idea is taken from \cite{S2014} where only monotonic independence of at most two algebras is obtained, we will only produce two asymptotically monotonically independent semicircular distributions.  Recall from \cite{M2001} that the monotone central limit distributions are also semicircular operators and two operators $Z_1$ and $Z_2$ in a non-commutative probability space $(\A, \varphi)$ are said to be \emph{monotonically independent with respect to $\varphi$} if
\[
\varphi\left(Z_2^{m_1} Z_1^{k_1} Z_2^{m_2} \cdots Z_1^{m_n} Z_2^{k_n} Z_1^{m_{n+1}} \right) = \varphi\left(Z_1^{k_1 + \cdots + k_n}\right) \prod^{n+1}_{j=1} \varphi\left(Z_2^{m_j}\right).
\]
for all $n \in \bN \cup \{0\}$, $m_1, m_{n+1} \in \bN \cup \{0\}$, and $k_j, m_j \in \bN$.

Three types of operators will need to be considered.  Again let $\X = \F(\H)$.  Fix $N \in \bN$ and let $\{h_n \, \mid\, n \in \{0, 1, \ldots, N\}\}$ be an orthonormal set in a Hilbert space $\H$.  Let $s(h_k) = l(h_k) + l^*(h_k)$, let
\[
T_1(N) = L\left( \sum^{N-1}_{j=1}   s(h_0)   \otimes E_{j+1, j}  \right) \qqand S_1(N) = R\left( \sum^{N-1}_{j=1}  E_{j, j+1}  \right).
\]
In addition, let
\[
T_2(N) = L\left(\diag(s(h_1), \ldots, s(h_n)) \right).
\]
Let $s_1, s_2$ be two semicircular variables each of variance 1 that are monotonically independent with respect to a state $\psi$.

\begin{thm}
With the above notation the joint distribution of $T_1(N)S_1(N)$ and $T_2(N)$ with respect to $\frac{1}{N} \tr \circ E_{\L(\X_N)}$ asymptotically tend to the distribution of $\{s_1, s_2\}$ with respect to $\psi$.
\end{thm}
\begin{proof}
It suffices to show that
\begin{align*}
\lim_{N \to \infty} & \frac{1}{N} \tr\left(E_{\L(\X_N)}\left( T_2(N)^{m_1} (T_1(N)S_1(N))^{k_1} T_2(N)^{m_2}  \cdots T_2(N)^{m_n} (T_1(N)S_1(N))^{k_n}  T_2(N)^{m_{n+1}} \right)     \right) \\
&= \psi\left(s_2^{m_1} s_1^{k_1} s_2^{m_2} \cdots s_1^{m_n} s_2^{k_n} s_1^{m_{n+1}}\right)
\end{align*}
where $n, m_1, m_{n+1} \in \bN \cup \{0\}$, and $k_j, m_j \in \bN$.

Recall the definition of monotonic independence implies
\[
\psi\left(s_2^{m_1} s_1^{k_1} s_2^{m_2} \cdots s_1^{m_n} s_2^{k_n} s_1^{m_{n+1}}\right) = \psi\left(s_1^{k_1 + \cdots + k_n}\right) \prod^{n+1}_{j=1} \psi\left(s_2^{m_j}\right).
\]
We claim that
\[
T_2(N)^{m_1} (T_1(N)S_1(N))^{k_1} T_2(N)^{m_2}  \cdots T_2(N)^{m_n} (T_1(N)S_1(N))^{k_n}  T_2(N)^{m_{n+1}}  I_{N, \Omega}
\]
is a diagonal matrix $\diag(z_1, \ldots, z_N)$ where
\[
z_x  = \left\{
\begin{array}{ll}
0 & \mbox{if } x \leq \sum^n_{j=1} k_j  \\
 s(h_{y+n})^{m_{1}}  s(h_0)^{k_1}   s(h_{y+n-1})^{m_{2}}   \cdots   s(h_{y+1})^{m_{n}}  s(h_0)^{k_n}   s(h_{y})^{m_{n+1}}   & \mbox{if } x = y + \sum^n_{j=1} k_j
\end{array} \right. .
\]
Indeed the above can be verified since
\[
T_2(N)^{m} \diag(\xi_1, \ldots, \xi_N) = \diag(s(h_1)^m\xi_1, \ldots,  s(h_N)^m\xi_N)
\]
and
\[
(T_1(N)S_1(N)) \diag(\xi_1, \ldots, \xi_N) = \diag(0, s(h_0) \xi_1,  s(h_0) \xi_2, \ldots, s(h_0) \xi_{N-1}).
\]
for all $\diag(\xi_1, \ldots, \xi_N) \in \X_N$.  Since
\[
\varphi\left(s(h_{y+n})^{m_{1}}  s(h_0)^{k_1}   \cdots   s(h_{y+1})^{m_{n}}  s(h_0)^{k_n}   s(h_{y})^{m_{n+1}}\right) = \varphi\left(s(h_0)^{k_1 + \cdots + k_n}\right) \prod^{n+1}_{j=1} \varphi\left(s(h_{y+n+1-j})^{m_j}  \right)
\]
due to freeness, and since each $s(h_k)$ is a semicircular operator of variance 1, the result now follows as the zero diagonal terms are irrelevant as $N \to \infty$.
\end{proof}

\section*{Acknowledgements}

The author would like to thank Dimitri Shlyakhtenko for referring him to the non-commutative case and the papers \cites{S1995, S1997}.

\end{document}